\documentclass[11pt]{amsart}
\usepackage{amsmath,amssymb,latexsym,amsmath,amsthm,amscd}
\usepackage[left=3cm,top=4cm,right=3cm,bottom = 4cm]{geometry}
\usepackage{xcolor}

\theoremstyle{plain}
\newtheorem{thm}{Theorem}
\newtheorem{lem}[thm]{Lemma}
\newtheorem{prop}[thm]{Proposition}
\newtheorem{cor}[thm]{Corollary}

\theoremstyle{definition}
\newtheorem{dfn}[thm]{Definition}

\theoremstyle{remark}

\newcommand{\End}{{\rm End}\,}
\newcommand{\Res}{{\rm Res}\,}

\newcommand{\Ind}{{\rm Ind}\,}

\def \<{\langle}
\def \>{\rangle}
\begin{document}

\title{On $\mathbb{N}$-graded vertex algebras associated with cyclic Leibniz algebras with small dimensions}
\author{C. Barnes}
\address{Department of Mathematics\\ Illinois State University, Normal, IL, 61790, USA}
\email{cnbarn1@ilstu.edu}
\author{E. Martin}
\address{Department of Mathematics\\ Illinois State University, Normal, IL, 61790, USA}
\email{ermart2@ilstu.edu}
\author{J. Service}
\address{Department of Mathematics\\ Illinois State University, Normal, IL, 61790, USA}
\email{jbserv1@ilstu.edu}
\author{G.Yamskulna}
\address{Department of Mathematics\\ Illinois State University, Normal, IL, 61790, USA}
\email{gyamsku@ilstu.edu}
\thanks{G. Yamskulna is supported by the College of Arts and Sciences, Illinois State University.}
\keywords{Indecomposable, irrational vertex algebras}
\thanks{Corresponding author: G. Yamskulna}

\begin{abstract} The main goals for this paper is i) to study of an algebraic structure of $\mathbb{N}$-graded vertex algebras $V_B$ associated to vertex $A$-algebroids $B$ when $B$ are cyclic non-Lie left Leibniz algebras, and ii) to explore relations between the vertex algebras $V_B$ and the rank one Heisenberg vertex operator algebra. To achieve these goals, we first classify vertex $A$-algebroids $B$ associated to given cyclic non-Lie left Leibniz algebras $B$. Next, we use the constructed vertex $A$-algebroids $B$ to create a family of indecomposable non-simple vertex algebras $V_B$. Finally,  we use the algebraic structure of the unital commutative associative algebras $A$ that we found to study relations between a certain type of the vertex algebras $V_B$ and the vertex operator algebra associated to a rank one Heisenberg algebra.  \end{abstract}
\maketitle

\section{Introduction}

A study of $\mathbb{N}$-graded vertex algebras $V=\oplus_{n=0}^{\infty}V_n$ when $\dim V_0\geq 2$ is far from being completed. For this case, $V_0$ is a unital commutative associative algebra, and $V_1$ is a Leibniz algebra. The skew symmetry and Jacobi identity of vertex algebra give rise to compatible extra relations. Indeed, these additional relations on $V_0\oplus V_1$ are summarized in the notion of a vertex $V_0$-algebroid $V_1$ (see \cite{GMS, MS1, MS2, MSV}). It was shown in \cite{GMS} that for a given vertex $A$-algebroid $B$, one can construct an $\mathbb{N}$-graded vertex algebra $V=\oplus_{n=0}^{\infty}V_n$ such that $V_0=A$ and $V_1=B$. Also, the classification of graded simple twisted and non-twisted modules of vertex algebras associated with vertex algebroids were studied in \cite{LiY, LiY2} by Li and the last author of this paper. 

In \cite{JY}, among many things, Jitjankarn and the last author of this paper show that for an $\mathbb{N}$-graded vertex algebra $V=\oplus_{n=0}^{\infty}V_n$ such that $2\leq \dim V_0<\infty$ and $1<\dim V_1<\infty$, if $V_0$ is a local algebra then $V$ is an indecomposable vertex algebra. Also, if $V$ is generated by $V_0$ and $V_1$, and $V_0$ is a local algebra and is not a simple module for a Lie $V_0$-algebroid $V_1/(V_0)_{-1}\partial(V_0)$, then $V$ is an indecomposable non-simple vertex algebra. In \cite{JY2}, Jitjankarn and the last author of this paper examined an algebraic structure of an indecomposable non-simple $\mathbb{N}$-graded vertex algebra associated with a vertex $A$-algebroid $B$ such that $B$ is (semi-)simple Leibniz algebra that has $sl_2$ as its Levi factor. The one-to-one correspondence between the set of representatives of equivalence classes of simple $sl_2$-modules and the set of representatives of equivalence classes of the $\mathbb{N}$-graded simple $V_B$-modules was established. In addition, Jitjankarn and the last of author of this paper showed that a certain quotient space of $V_B$ is an indecomposable non-simple vertex algebra that satisfies the $C_2$-condition and has only two irreducible modules. These two irreducible modules are in fact isomorphic to irreducible modules of the rational $C_2$-cofinite $CFT$-type vertex operator algebra $V_{\mathbb{Z}\alpha}$ associated with a rank one lattice $\mathbb{Z}\alpha$ equipped with a bilinear form $(~,~)$ such that $(\alpha,\alpha)=2$. In \cite{BuY}, Bui and the last author of this paper generalized the results in \cite{JY2} to the case when a vertex $A$-algebroid $B$ is an arbitrary finite-dimensional simple Leibniz algebras. Indecomposable non-simple $C_2$-cofinite $\mathbb{N}$-graded vertex algebras $U_B$ were established. A one-to-one correspondence between these irreducible $U_B$-modules and irreducible modules of certain types of rational $C_2$-cofinite $CFT$-type affine vertex algebras was exhibited as well. Note that the results in \cite{BuY, JY2} provided new families of irrational $C_2$-cofinite vertex algebras.

The work in this paper is a stepping stone of our attempt to understand an algebraic structure $\mathbb{N}$-graded vertex algebras $V_B$ associated to vertex $A$-algebroids $B$ when $B$ are no longer (semi) simple Leibniz algebras and their connections to many well known $\mathbb{N}$-graded vertex algebras $V=\oplus_{n=0}^{\infty}V_n$ such that $\dim V_0=1$.  We are also interested in roles of the commutative associative algebras $A$ that play on a study of representation theory of $\mathbb{N}$-graded vertex algebras.

In this paper, we construct indecomposable non-simple $\mathbb{N}$-graded vertex algebras $V_B$ associated to vertex $A$-algebroids $B$ when $B$ are non-Lie cyclic left Leibniz algebra. In addition we investigate relationships between the constructed indecomposable non-simple vertex algebras $V_B$ and a rank one Heisenberg vertex operator algebra. As mentioned above, to study indecomposable non-simple properties of $\mathbb{N}$-graded vertex algebras $V=\oplus_{n=0}^{\infty}V_n$ that are generated by $V_0$ and $V_1$, we only need to study the algebraic structure of vertex $V_0$-algebroid $V_1$. Therefore, it is natural to begin our study by classifying and investigating an algebraic structure of vertex $A$- algebroids $B$ associated to non-Lie cyclic Leibniz algebras $B$. This classification of vertex $A$-algebroids $B$ help us identify connections between the indecomposable non-simple vertex algebra $V_B$ and a rank one Heisenberg vertex operator algebra. In fact, commutative associative algebras $A$ are the main tools for pining down possible relations between the indecomposable non-simple vertex algebras $V_B$ and a rank one Heisenberg vertex operator algebra. We show that appropriate quotient spaces of a certain type of indecomposable, non-simple vertex algebras associated to vertex $A$-alebroids $B$ are in fact a rank one Heisenberg vertex operator algebras.

In Section 2, we provide background on cyclic left Leibniz algebras and vertex algebroids. In Section 3, we classify vertex algebroids associated to cyclic left Leibniz algebras. Precisely, in subsections 3.1 and 3.2, we classify vertex algebroids associate to 2-dimensional non-Lie cyclic left Leibniz algebras and 3-dimensional non-Lie cyclic left Leibniz algebra, respectively. A review for a construction of vertex algebras associated to vertex algebroids and their modules is provided in subsection 4.1. Discussions about representation theory of vertex algebras associated to vertex $A$-algebroids and their modules when $B$ are cyclic non-Lie left Liebniz algebras are in subsection 4.2. Exploration of relationships between these vertex algebras and the vertex operator algebra associated with a rank one Heisenberg algebra is in subsection 4.3. 
\section{A review on cyclic Leibniz algebras and vertex algebroids}

We begin by providing background on non-Lie cyclic Leibniz algebras that we will use in this paper. After that we recall definitions of a 1-truncated conformal algebra, a vertex algebroid, a Lie algebroid, and discuss about modules of Lie algebroids. 
\begin{dfn}\cite{DMS, FM} A {\em left Leibniz algebra} $\mathfrak{L}$ is a $\mathbb{C}$-vector space equipped with a bilinear map $[~,~]:\mathfrak{L}\times\mathfrak{L}\rightarrow\mathfrak{L}$ satisfying the Leibniz identity $$[a,[b,c]]=[[a,b],c]+[b,[a,c]]$$ for all $a,b,c\in\mathfrak{L}$.

Let $I$ be a subspace of $\mathfrak{L}$. $I$ is a {\em left} (respectively, {\em right}) {\em ideal} of $\mathfrak{L}$ if $[\mathfrak{L}, I]\subseteq I$ (respectively, $[I,\mathfrak{L}]\subseteq I$). $I$ is an {\em ideal} of $\mathfrak{L}$ if it is both a left and a right ideal. 
\end{dfn}

Let $\mathfrak{L}$ be a left Leibniz algebra and let $u\in \mathfrak{L}$. Now, we fix the notation $u^1=u$, $u^2=[u,u]$, and in general, $u^{n+1}=[u,u^n]$ for $n\geq 1$. Clearly, $[[u,u],y]=0$, and more generally $[u^n,y]=0$ for all $y\in \mathfrak{L}$, $n\geq 2$.

\begin{dfn} Let $\mathfrak{L}$ be a Leibniz algebra. $\mathfrak{L}$ is cyclic if and only if there exists some $u\in\mathfrak{L}$ such that $\mathfrak{L}=\langle u\rangle=Span\{u^k~|~k=1,2,....\}$. If $\mathfrak{L}=\langle v\rangle$, we call $v$ a generator of $\mathfrak{L}$. 
\end{dfn}

\begin{prop}\label{cyclicclass}\cite{BuHLSS, DMS}

\ \ 

\begin{enumerate} 

\item For a non-Lie left Leibniz algebra $B$ such that $\dim B=2$, $B$ is isomorphic to a cyclic left Leibniz algebra generated by $b$ with either $[b,b^2]=0$ or $[b,b^2]=b^2$. 

\item A complete list of cyclic non-Lie left Leibniz algebras $B$ such that $\dim B=3$ with non-zero bracket $[~,~]$ is the following:
\begin{enumerate}
    \item $[x,x]=y$, $[x,y]=z$.
    \item $[z,y]=y$; $[z,x]=\alpha x$, $\alpha\in\mathbb{C}\backslash\{0\}$ and $\alpha\neq 1$.
    \item $[z,x]=x+y$; $[z,y]=y$.
    \item $[z,x]=y$; $[z,y]=y$; $[z,z]=x$.
\end{enumerate}
\end{enumerate}
\end{prop}
\begin{cor}\label{cyclicclass3}\cite{BuHLSS} 

Let $B$ be a non-Lie cyclic Leibniz algebra $B$ such that $\dim B=3$. 
\begin{enumerate}
\item If $B$ is of type $(b)$ in Proposition \ref{cyclicclass}, then $B$ has a basis $\{b,b^2,b^3\}$ such that $b^4=b^2-\frac{\alpha+1}{\sqrt{\alpha}}ib^3$ ($\alpha\neq 1$).
\item If $B$ is of type $(c)$ in Proposition \ref{cyclicclass}, then $B$ has a basis $\{b,b^2,b^3\}$ such that $b^4=b^2+2ib^3$.
\item If $B$ is of type $(d)$ in Proposition \ref{cyclicclass}, then $B$ has a basis $\{b,b^2,b^3\}$ such that $b^4=b^3$.
\end{enumerate}

\end{cor}


\begin{dfn}\cite{GMS} A {\em 1-truncated conformal algebra} is a graded vector space $C=C_0\oplus C_1$ equipped with a linear map $\partial:C_0\rightarrow C_1$ and bilinear operations $(u,v)\mapsto u_iv$ for $i=0,1$ of degree $-i-1$ on $C=C_0\oplus C_1$ such that the following axioms hold:

\medskip

\noindent(Derivation) for $a\in C_0$, $u\in C_1$,
\begin{equation}
(\partial a)_0=0,\ \ (\partial a)_1=-a_0,\ \ \partial(u_0a)=u_0\partial a;
\end{equation}

\noindent(Commutativity) for $a\in C_0$, $u,v\in C_1$,
\begin{equation} 
u_0a=-a_0u,\ \ u_0v=-v_0u+\partial(u_1v),\ \ u_1v=v_1u;
\end{equation}

\noindent(Associativity) for $\alpha,\beta,\gamma\in C_0\oplus C_1$,
\begin{equation}
\alpha_0\beta_i\gamma=\beta_i\alpha_0\gamma+(\alpha_0\beta)_i\gamma.
\end{equation}
\end{dfn}

\begin{dfn}\cite{Br1, Br2, GMS} Let $(A,*)$ be a unital commutative associative algebra over $\mathbb{C}$ with the identity $1_A$. A {\em vertex $A$-algebroid} is a $\mathbb{C}$-vector space $\Gamma$ equipped with 
\begin{enumerate}
\item a $\mathbb{C}$-bilinear map $A\times \Gamma\rightarrow \Gamma, \ \ (a,v)\mapsto a\cdot v$ such that $1\cdot v=v$ (i.e. a nonassociative unital $A$-module),
\item a structure of a Leibniz $\mathbb{C}$-algebra $[~,~]:\Gamma\times \Gamma\rightarrow\Gamma$, 
\item a homomorphism of left Leibniz $\mathbb{C}$-algebra $\pi:\Gamma\rightarrow Der(A)$,
\item a symmetric $\mathbb{C}$-bilinear pairing $\langle ~,~\rangle:\Gamma\otimes_{\mathbb{C}}\Gamma\rightarrow A$,
\item a $\mathbb{C}$-linear map $\partial :A\rightarrow \Gamma$ such that $\pi\circ \partial =0$ which satisfying the following conditions:
\begin{eqnarray*}
&&a\cdot (a'\cdot v)-(a*a')\cdot v=\pi(v)(a)\cdot \partial(a')+\pi(v)(a')\cdot \partial(a),\\
&&[u,a\cdot v]=\pi(u)(a)\cdot v+a\cdot [u,v],\\
&&[u,v]+[v,u]=\partial(\langle u,v\rangle),\\
&&\pi(a\cdot v)=a\pi(v),\\
&&\langle a\cdot u,v\rangle=a*\langle u,v\rangle-\pi(u)(\pi(v)(a)),\\
&&\pi(v)(\langle v_1,v_2\rangle)=\langle [v,v_1],v_2\rangle+\langle v_1,[v,v_2]\rangle,\\
&&\partial(a*a')=a\cdot \partial(a')+a'\cdot\partial(a),\\
&&[v,\partial(a)]=\partial(\pi(v)(a)),\\
&&\langle v,\partial(a)\rangle=\pi(v)(a)
\end{eqnarray*}
for $a,a'\in A$, $u,v,v_1,v_2\in\Gamma$.
\end{enumerate}
\end{dfn}


\begin{prop}\cite{LiY} Let $(A,*)$ be a unital commutative associative algebra and let $B$ be a module for $A$ as a nonassociative algebra . Then a vertex $A$-algebroid structure on $B$ exactly amounts to a 1-truncated conformal algebra structure on $C=A\oplus B$ with 
\begin{eqnarray*}
&&a_ia'=0,\\
&&u_0v=[u,v],~u_1v=\langle u,v\rangle,\\
&&u_0a=\pi(u)(a),~ a_0u=-u_0a
\end{eqnarray*} for $a,a'\in A$, $u,v\in B$, $i=0,1$ such that 
\begin{eqnarray*}
&&a\cdot(a'\cdot u)-(a*a')\cdot u=(u_0a)\cdot \partial a'+(u_0a')\cdot \partial a,\\
&&u_0(a\cdot v)-a\cdot (u_0v)=(u_0a)\cdot v,\\
&&u_0(a*a')=a*(u_0a')+(u_0a)*a',\\
&&a_0(a'\cdot v)=a'*(a_0v),\\
&&(a\cdot u)_1v=a*(u_1v)-u_0v_0a,\\
&&\partial(a*a')=a\cdot \partial(a')+a'\cdot \partial(a).
\end{eqnarray*}
\end{prop}
\begin{dfn} Let $I$ be a subspace of a vertex $A$-algebroid $B$. The vector space $I$ is called an {\em ideal } of the vertex $A$-algebroid $B$ if $I$ is a left ideal of the left Leibniz algebra $B$ and $a\cdot u\in I$ for all $a\in A$, $u\in I$
\end{dfn}


Let $(A,*)$ be a unital commutative associative algebra. Let $B$ be a vertex $A$-algebroid. We set $A\partial(A):=Span\{a\cdot\partial(a')~|~a,a'\in A\}$. The vector space $A\partial(A)$ is an ideal of the vertex $A$-algebroid $B$. Moreover, $A\partial(A)$ is an abelian Lie algebra.

Observe that for $a,a',a''\in A$, $u\in B$, we have
\begin{eqnarray*}
(a\cdot \partial(a'))_0a''&&=0,\text{ and }\\
(a\cdot \partial(a'))_0u&&=a\cdot (u_0\partial(a'))+(u_0a)\cdot\partial(a')+\partial(u_1( a\cdot \partial(a'))\in A\partial(A).
\end{eqnarray*}

\begin{dfn} Let $A$ be a commutative associative algebra. A Lie $A$-algebroid is a Lie algebra $\mathfrak{g}$ equipped with an $A$-module structure and a module action on $A$ by derivation such that 
\begin{eqnarray*}
[u,av]&=&a[u,v]+(ua)v,\\
a(ua')&=&(au)a'
\end{eqnarray*}
for $u,v\in\mathfrak{g}$, $a,a'\in A$.

A module for a Lie $A$-algebroid $\mathfrak{g}$ is a vector space $W$ equipped with a $\mathfrak{g}$-module structure and an $A$-module structure such that 
\begin{eqnarray*}
&&u(aw)-a(uw)=(ua)w,\\
&&a(uw)=(au)w
\end{eqnarray*}
for all $a\in A$, $u\in\mathfrak{g}$, $w\in W$.
\end{dfn}
\begin{prop}\cite{Br2}  Let $A$ be a unital commutative associative algebra over $\mathbb{C}$ and let $B$ be a vertex $A$-algebroid . Then $B/A \partial(A)$ is a Lie $A$-algebroid. 
\end{prop}
\begin{prop}\cite{LiY} Let $\mathfrak{g}$ be a Lie $A$-algebroid. Then $\mathfrak{g}$ is a Lie algebra with $A$ a $\mathfrak{g}$-module. By adjoining the $\mathfrak{g}$-module $A$ to $\mathfrak{g}$ we have a Lie algebra $A\oplus\mathfrak{g}$ with $A$ as an abelian ideal. Denote by $J$ the 2-sided ideal of the universal enveloping algebra $U(A\oplus\mathfrak{g})$ generated by the vectors $$1_A-1,~a\cdot a'-aa',~ a\cdot b-ab$$ for $a,a'\in A$, $b\in\mathfrak{g}$ where $1_A$ is the identity of $A$, and $\cdot $ denotes the product in the universal enveloping algebra. Set $$\overline{U}(A\oplus \mathfrak{g})=U(A\oplus \mathfrak{g})/J.$$ Then a (simple) module structure for the Lie $A$-algebroid $\mathfrak{g}$ on a vector space $W$ exactly amounts to a (simple) $\overline{U}(A\oplus\mathfrak{g})$-module structure. 
\end{prop}

\section{Classification of vertex algebroids associated with non-Lie cyclic left Leibniz algebras}

In this section, we classify all vertex $A$-algebroids $B$ when $B$ are cyclic non-Lie left Leibniz algebras such that $B\neq A\partial(A)$ and $\dim B$ is either 2 or 3. In subsection 3.1, we will find all vertex $A$-algebroids $B$ when $B$ are 2-dimensional non-Lie cyclic left Leibniz algebras such that $B\neq A\partial(A)$. In subsection 3.2, we investigate and classify all vertex $A$-algebroids $B$ when $B$ are 3-dimensional cyclic non-Lie left Leibniz algebras such that $B\neq A\partial(A)$. 

It is worth mentioning that the results in this section will play a crucial role in a study of representation theory of vertex algebras $V_B$ associated to vertex $A$-algebroids $B$ when $B$ is a non-Lie cyclic left Leibniz algebra such that $B\neq A\partial(A)$ and $2\leq \dim B\leq 3$, and an investigation on the relationships between vertex algebras $V_B$ and a rank one Heisenberg vertex operator algebra in Section 4.

Now, we let $A$ be a unital commutative associative algebra, and let $B$ be a vertex $A$-algebroid such that $B$ is a cyclic non-Lie left Leibniz algebra, $B\neq A\partial(A)$, and $\dim (B)=n$. Hence, $B$ is of the form $Span\{b,b_0b,....., (b_0)^{n-1}b\}$. We set $a=b_1b$. Since $b_0b=-b_0b+\partial(b_1b)$ and $B$ is a cyclic non-Lie left Leibniz algebra, we can conclude that $\partial(a)=2b_0b\neq 0$ and $a\neq 0$. Using the fact that $\{b, b_0b,.....,(b_0)^{n-1}b\}$ is a basis of $B$, and $(b_0)^i\partial(a)=\partial((b_0)^ia)$, we can conclude that $$\{b, \partial(a),...,\partial((b_0)^{n-2}a)\}$$ is a basis of $B$ as well. Since $\partial(A)\subseteq A\partial(A)\subseteq B$, and $1=\dim B/\partial(A)\geq\dim B/A\partial(A)\geq 1 $, we can conclude that $A\partial(A)=\partial(A)$, and 
\begin{equation*}\{\partial(a),...,\partial((b_0)^{n-2}a) \}\end{equation*} is a basis for $\partial(A)$.

\begin{prop}\cite{JY} Let $A$ be a finite dimensional unital commutative associative algebra with the identity $1_A$ and let $B$ be a finite dimensional non-Lie left Leibniz algebra. Assume that $B$ is a vertex $A$-algebroid. If $Ker(\partial)=\mathbb{C}1_A$ then $A$ is a local algebra.
\end{prop}
\begin{prop} Let $A$ be a unital commutative associative algebra with the identity $1_A$. Let $B$ be a vertex $A$-algebroid such that $B$ is a cyclic non-Lie left Leibniz algebra, $B\neq A \partial(A)$, and $\dim (B)=n$. Then, there exists $b\in B$ such  that  $\{b,b_0b,....., (b_0)^{n-1}b\}$ is a basis for $B$. We set $a=b_1b$. Assume that $Ker(\partial)=\mathbb{C}1_A$. Then $A$ is a local algebra with a basis $\{1_A, a,b_0a,....,(b_0)^{n-2}a\}$.
\end{prop}
\begin{proof} First, we will show that $\{1_A,a,b_0a,....,(b_0)^{n-2}a\}$ is linearly independent. We set 
\begin{equation}\label{Aindependent}\alpha 1_A+\alpha_0 a+\alpha_1 b_0a+.....+\alpha_{n-2}(b_0)^{n-2}a=0.\end{equation} Applying $\partial$ to (\ref{Aindependent}), we have $$\alpha_0\partial(a)+\alpha_1\partial(b_0a)+....+\alpha_{n-2}\partial((b_0)^{n-2}a)=0.$$ Since $\{\partial(a),\partial(b_0a),....,\partial((b_0)^{n-2}a)\}$ is linearly independent, we can conclude that 
$\alpha_0=\alpha_1=....=\alpha_{n-2}=0$. In addition, we have $\alpha 1_A=0$. This implies that $\alpha=0$, and $\{1_A,a,b_0a,....,(b_0)^{n-2}a\}$ is linearly independent. 

Next, we show that $A=Span\{1_A,a,b_0a,....,(b_0)^{n-2}a\}$. Let $a'\in A$. Then $\partial(a')\in \partial(A)$. Hence, there exist $\beta_0,....,\beta_{n-2}\in\mathbb{C}$ such that $$\partial(a')=\beta_0\partial(a)+\beta_1\partial(b_0a)+....+\beta_{n-2}\partial((b_0)^{n-2}a).$$ Consequently, we have $a'-(\beta_0a+\beta_1 b_0a+....+\beta_{n-2}(b_0)^{n-2}a) \in Ker(\partial).$ Since $Ker(\partial)=\mathbb{C}1_A$, this implies that there exists $\chi\in \mathbb{C}$ such that $$a'=\beta_0a+\beta_1 b_0a+....+\beta_{n-2}(b_0)^{n-2}a+\chi 1_A.$$ Therefore, $A=Span\{1_A,a,b_0a,....,(b_0)^{n-2}a\}$. 

\end{proof}

\subsection{Classification of Vertex Algebroids associated to $2$-dimensional non-Lie left Leibniz algebras}

\ \ 

Let $A$ be a unital commutative associative algebra with the identity $1_A$. Let $B$ be a vertex $A$-algebroid such that $B$ is a non-Lie left Leibniz algebra, $\dim(B)=2$, and $B\neq A\partial(A)$. Hence, there exists $b\in B$ such that $\{b,b_0b\}$ is a basis of $B$. 

We set $a=b_1b$.  Assume that $Ker(\partial)=\mathbb{C}1_A$. Then $A$ is a local algebra with a basis $\{1_A,a\}$, and $\{b,\partial(a)\}$ is a basis for $B$. For convenience, we also set $$a*a=\alpha_1 1_A+\alpha_2 a.$$ Here, $\alpha_1,\alpha_2\in\mathbb{C}$. Clearly, we have $$a\cdot \partial(a)=\frac{1}{2}\alpha_2\partial(a).$$ Since $B$ is a 2-dimensional left Leibniz algebra, by Proposition \ref{cyclicclass}, we know that either $b_0(b_0b)=0$ or $b_0(b_0b)=b_0b$. In Theorem \ref{Bcyclicnilpotent} and Theorem \ref{Bcyclicsolvable}, we study the algebraic structure of the vertex $A$-algebroid $B$ for the case when $b_0(b_0b)=0$ and for the case when $b_0(b_0b)=b_0b$, respectively.  

\begin{thm}\label{Bcyclicnilpotent} Let $A$ be a unital commutative associative algebra with the identity $1_A$. Let $B$ be a vertex $A$-algebroid such that $B$ is a cyclic non-Lie left Leibniz algebra, $\dim B=2$, $B\neq A \partial(A)$, and $Ker(\partial)=\mathbb{C}1_A$.  Clearly, there exists $b\in B$ such that $\{ b,b_0b\}$ is a basis for $B$. When we set $a=b_1b$, the set $\{b,\partial(a)\}$ is a basis of $B$ and the set $\{1_A,a\}$ is a basis for $A$. 

Assume that $b_0(b_0b)=0$. Then

\noindent(i) $b_0a=0$,
   
\noindent(ii) $a\cdot b\in\partial(A)$,
   
\noindent(iii) $a*a=0$, $A\cong \mathbb{C}[x]/(x^2)$, and

\noindent(iv) $a\cdot\partial(a)=0$. 
    
\noindent (v) In addition, $B$ is a module of $A$ as a commutative associative algebra.

\noindent (vi) The ideal $(a)$ is the unique maximal ideal of $A$. Moreover, for $u\in (a)$, $w\in B/A\partial(A)$, $w_0u=0$ and $u\cdot w=0$.
\end{thm}
\begin{proof} Assume that $b_0b_0b=0$. Because $\partial(a)=2b_0b$, we have $b_0\partial(a)=0$. First, we will prove statement (i). Since $\partial(b_0a)=b_0\partial(a)=0$, we then have that $b_0a\in Ker(\partial)=\mathbb{C}1_A$. So, there exists $\lambda\in\mathbb{C} $ such that $b_0a=\lambda 1_A$. Moreover, $b_0b_0a=0.$

We set $a\cdot b=\beta_1 b+\beta_2 \partial(a)$. Recall that $a\cdot \partial(a)=\frac{1}{2}\alpha_2\partial(a) $. Since 
\begin{eqnarray*}
b_0(a\cdot b)&&=a\cdot b_0b+(b_0a)\cdot b=\frac{1}{2}a\cdot \partial(a)+\lambda b=\frac{1}{4}\alpha_2\partial(a)+\lambda b\text{ and}\\
b_0(a\cdot b)&&=b_0(\beta_1 b+\beta_2 \partial(a) )=\beta_1 \frac{1}{2}\partial(a)+\beta_2 \partial(b_0a)=\beta_1\frac{1}{2}\partial(a),
\end{eqnarray*}
we have $\beta_1\frac{1}{2}\partial(a)=\frac{1}{4}\alpha_2\partial(a)+\lambda b$.
Because $\{b,\partial(a)\}$ is linearly independent, we can conclude that $\lambda=0$ and $\frac{1}{4}\alpha_2=\frac{1}{2}\beta_1$. In addition, we have 
\begin{equation}\label{adotbB2}
b_0a=0,~\beta_1=\frac{1}{2}\alpha_2\text{ and }a\cdot b=\frac{1}{2}\alpha_2 b+\beta_2 \partial(a).\end{equation} This proves $(i)$.

Next, we will proof statement $(v)$. Recall that for $a',a''\in A$, $u\in B$, $$a''\cdot (a'\cdot u)-(a''*a')\cdot u=(u_0a'')\cdot \partial(a')+(u_0a')\cdot \partial(a'').$$ If we set $u=b$, $a'=a''=a$, then
$$a\cdot(a\cdot b)-(a*a)\cdot b=(b_0a)\cdot \partial(a)+(b_0a)\cdot \partial(a).$$ Since $b_0a=0$, we have 
\begin{equation}\label{aab} a\cdot (a\cdot b)-(a*a)\cdot b=0.\end{equation} 
Similarly, if we set $u=\partial(a)$, $a'=a''=a$, then
\begin{equation}\label{aapartiala} a\cdot (a\cdot \partial(a))-(a*a)\cdot \partial(a)=((\partial(a))_0a)\cdot \partial(a)+((\partial(a))_0a)\cdot\partial(a)=0. \end{equation}
By (\ref{aab}), (\ref{aapartiala}), we can conclude that $B$ is a module of $A$ as a commutative associative algebra. This proves $(v)$.

Now, we will prove statements $(ii)-(iv)$. Recall that for $u,v\in B$, $a\in A$, $$(a\cdot u)_1v=a*(u_1v)-u_0v_0a.$$ When we set $u=v=b$, we have 
$(a\cdot b)_1b=a*(b_1b)-b_0b_0a$. Since $b_1b=a$ and $b_0a=0$, we obtain that 
\begin{equation}\label{abaa}
(a\cdot b)_1b=a*a.
\end{equation} 
By (\ref{adotbB2}), we have $$a*a=(a\cdot b)_1b=(\frac{1}{2}\alpha_2 b+\beta_2 \partial(a) )_1b=\frac{1}{2}\alpha_2 a+\beta_2(b_0a)=\frac{1}{2}\alpha_2 a.$$
Since $a*a=\alpha_1 1_A+\alpha_2 a$ and $a*a=\frac{1}{2}\alpha_2 a$, we can conclude that 
$\alpha_1 1_A+\frac{1}{2}\alpha_2 a=0$. This implies that $\alpha_1=0$, and $\alpha_2=0$. In addition, $$a*a=0,~a\cdot b=\beta_2\partial(a)\in \partial(A),~a\cdot\partial(a)=\frac{1}{2}\alpha_2\partial(a)=0.$$ This proves $(ii)-(iv)$. The statements $(v)$, $(vi)$ are consequences of $(i)$-$(iv)$.
\end{proof}

\begin{thm}\label{Bcyclicsolvable} Let $A$ be a unital commutative associative algebra with the identity $1_A$. Let $B$ be a vertex $A$-algebroid such that $B$ is a cyclic non-Lie left Leibniz algebra, $\dim B=2$, $B\neq A\partial(A)$, and $Ker(\partial)=\mathbb{C}1_A$.  There exists $b\in B$ such that $\{ b,b_0b\}$ is a basis for $B$. We set $a=b_1b$. Then $\{b,\partial(a)\}$ is a basis of $B$ and $\{1_A,a\}$ is a basis for $A$. 

Assume that $b_0(b_0b)=b_0b$. If we set $a*a=\alpha_1 1_A +\alpha_2 a$, then we have

\noindent(i) $\alpha_1=-\frac{1}{4}\alpha_2^2$,

\noindent(ii) $b_0a=a-\frac{1}{2}\alpha_2 1_A$,

\noindent(iii) $a\cdot b=\frac{1}{2}\alpha_2 b+(\frac{1}{2}\alpha_2-1)\partial(a)$,

\noindent(iv) $a*a=\alpha_2 a-\frac{1}{4}\alpha_2^2 1_A$, $a\cdot\partial(a)=\frac{\alpha_2}{2}\partial(a)$ and $A\cong \mathbb{C}[x]/(x-\frac{1}{2}\alpha_2)^2$.

\noindent (v) The vector space $(a-\frac{1}{2}\alpha_2 1_A )$ is the unique maximal ideal of $A$. In addition, for $u\in (a-\frac{1}{2}\alpha_2 1_A)$, $w\in B/A\partial(A)$, $w_0u\in (a-\frac{1}{2}\alpha_2 1_A)$ and $u\cdot w=0$.
\end{thm}

\begin{proof} Since $\partial(b_0a)=\partial(a)$, we then have that $\partial(b_0a-a)=0$, and $b_0a-a\in Ker(\partial)$. Since $Ker(\partial)=\mathbb{C} 1_A$, we can conclude that $b_0a=a+\rho 1_A$ for some $\rho\in\mathbb{C}$.

Now, we set $a\cdot b=\beta_1 b+\beta_2\partial(a)$. Since $b_0(a\cdot b)=a\cdot (b_0b)+(b_0a)\cdot b$, and $b_0a=a+\rho 1_A$, we then have that $b_0(a\cdot b)=a\cdot \frac{1}{2}\partial(a)+(a+\rho 1_A )\cdot b.$
Since 
\begin{eqnarray*}
b_0(a\cdot b)&&=b_0(\beta_1 b+\beta_2\partial(a) )\\
&&=\beta_1\frac{1}{2}\partial(a)+\beta_2\partial(b_0a)\\
&&=\frac{1}{2}\beta_1\partial(a)+\beta_2\partial(a)\\
&&=(\frac{1}{2}\beta_1+\beta_2)\partial(a)
\end{eqnarray*}
and 
\begin{eqnarray*}
b_0(a\cdot b)&&=a\cdot \frac{1}{2}\partial(a)+(a+\rho 1_A )\cdot b\\
&&=\frac{1}{4}\alpha_2\partial(a)+a\cdot b+\rho b\\
&&=\frac{1}{4}\alpha_2\partial(a)+\beta_1 b+\beta_2\partial(a)+\rho b\\
&&=(\beta_1+\rho)b+(\frac{1}{4}\alpha_2+\beta_2)\partial(a),
\end{eqnarray*}
we can conclude that $(\frac{1}{2}\beta_1+\beta_2)\partial(a)  = (\beta_1+\rho)b+(\frac{1}{4}\alpha_2+\beta_2)\partial(a).$ Because $\{b,\partial(a)\}$ is a basis, 
we have $\beta_1=-\rho\text{ and }\beta_1=\frac{1}{2}\alpha_2.$ These imply that $\rho=-\frac{1}{2}\alpha_2.$
In summary, we have 
\begin{eqnarray*}
b_0a&&=a-\frac{1}{2}\alpha_2 1_A,\\
a\cdot b&&=\frac{1}{2}\alpha_2 b+\beta_2\partial(a).
\end{eqnarray*} This proves $(ii)$.

Since $b_0(a*a)=b_0(\alpha_1 1_A+\alpha_2 a)=\alpha_2b_0a=\alpha_2(a-\frac{1}{2}\alpha_2 1_A)=\alpha_2 a-\frac{1}{2}\alpha_2^2 1_A$ and 
$$b_0(a*a)=a*(b_0a)+(b_0a)*a=2 a*(b_0a)=2 a*(a-\frac{1}{2}\alpha_2 1_A)= 2 a*a-\alpha_2 a,$$
we have  
$$a*a=\alpha_2 a-\frac{1}{4}\alpha_2^2 1_A.$$ This proves $(iv)$. 

Since $a*a=\alpha_1 1_A+\alpha_2 a$ and $\{1_A, a\}$ is a basis of $A$, we can conclude that 
$$\alpha_1=-\frac{1}{4}\alpha_2^2.$$ This proves $(i)$.

Now, we will solve for $\beta_2$. Recall that $(a\cdot b)_1b=a*a-b_0(b_0a) $. Since 
\begin{eqnarray*}
(a\cdot b)_1b&=&(\frac{1}{2}\alpha_2 b+\beta_2\partial(a) )_1b\\
&=&\frac{1}{2}\alpha_2 a+\beta_2b_0a,\\
&=&\frac{1}{2}\alpha_2 a+\beta_2(a-\frac{1}{2}\alpha_2 1_A )\text{ and }\\
(a\cdot b)_1b&=&a*a-b_0(b_0a)\\
&=&\alpha_2a-\frac{1}{4}\alpha_2^2 1_A-b_0(a-\frac{1}{2}\alpha_2 1_A )\\
&=&\alpha_2a-\frac{1}{4}\alpha_2^2 1_A-b_0a\\
&=&\alpha_2a-\frac{1}{4}\alpha_2^2 1_A-(a-\frac{1}{2}\alpha_2 1_A )\\
\end{eqnarray*}
we have 
$\frac{1}{2}\alpha_2 a+\beta_2(a-\frac{1}{2}\alpha_2 1_A )=\alpha_2a-\frac{1}{4}\alpha_2^2 1_A-(a-\frac{1}{2}\alpha_2 1_A ).$ 
Equivalently, we have
$$(\frac{1}{2}\alpha_2+(-1-\beta_2))a+(-\frac{1}{4}\alpha_2^2+(1+\beta_2)(\frac{1}{2}\alpha_2))1_A=0.$$
Using the fact that $\{1_A,a\}$ is linearly independent, we can conclude that 
$(\frac{1}{2}\alpha_2+(-1-\beta_2))=0$ and $(-\frac{1}{4}\alpha_2^2+(1+\beta_2)(\frac{1}{2}\alpha_2))=0.$ This implies that $$\beta_2=\frac{1}{2}\alpha_2-1.$$
Therefore, $a\cdot b=\frac{1}{2}\alpha_2 b+(\frac{1}{2}\alpha_2-1)\partial(a).$ This proves $(iii)$. The statement $(v)$ follows immediately from $(ii)$-$(iv)$.
\end{proof}

\subsection{Classification of Vertex Algebroids associated with $3$-dimensional cyclic non-Lie Leibniz algebras}

\ \ 

There are 4 types of cyclic non-Lie left Leibniz algebras (cf. Proposition \ref{cyclicclass} and Corollary \ref{cyclicclass3}). In this subsection, we will classify vertex algebroids associated with  these 4 types of 3-dimensional cyclic non-Lie left Leibniz algebras. The results of these classification are in Theorem \ref{3nilcase}-Theorem \ref{3solvablecase3}.

\begin{lem}\label{dim3case} Let $A$ be a unital commutative associative algebra with the identity $1_A$. Let $B$ be a vertex $A$-algebroid such that $B$ is a cyclic non-Lie left Leibniz algebra, $\dim B=3$, $B\neq A\partial(A)$, and $Ker(\partial)=\mathbb{C}1_A$. Then  there is $b\in B$ such that $\{b, b_0b, (b_0)^2b\}$ is a basis of $B$. In addition, if we set $a=b_1b$, then $\{1_A, a, b_0a\}$ is a basis of $A$, and $\{b,\partial(a),\partial(b_0a)\}$ is a basis of $B$. 

Assume that
\begin{eqnarray}
&&a\cdot b=\beta b+\gamma_0\partial(a)+\gamma_1\partial(b_0a),\label{3relab}\\
&&b_0\partial(b_0a)=c_0 \partial(a)+c_1 \partial(b_0a).\label{3relb0b0a}
\end{eqnarray} 
Here, $\beta, \gamma_0, \gamma_1, c_0, c_1\in \mathbb{C}$. Then we have the following statements: 
\begin{eqnarray*}
&&b_0b_0a=\chi 1_A+c_0a+c_1b_0a\text{ for some }\chi\in\mathbb{C},\\
&&a*(b_0a)= \beta b_0a,\\
&&(b_0a)*(b_0a)=0,\\
&&(b_0a)\cdot \partial(b_0a)=0,\\
&&(b_0a)\cdot b\in A\partial(A),\\
&&a*a=(\gamma_1+1)\chi 1_A+(\beta +(\gamma_1+1)c_0)a+(\gamma_0+(\gamma_1+1)c_1)b_0a,\\
&&a\cdot \partial(a)=\frac{1}{2} ((\beta +(\gamma_1+1)c_0)\partial(a)+(\gamma_0+(\gamma_1+1)c_1)\partial(b_0a)).
\end{eqnarray*}
Moreover, the following statements hold
\begin{eqnarray*}
&&(\gamma_0+(\gamma_1+1)c_1)\chi =0,\\
&&(\gamma_0+(\gamma_1+1)c_1 )c_0=0,\\
&&\beta=(\gamma_1+1)c_0+(\gamma_0+(\gamma_1+1)c_1)c_1 .
\end{eqnarray*}
\end{lem}
\begin{proof} Because $\partial((b_0)^2a-(c_0a+c_1b_0a))=0\text{ (cf. equation (\ref{3relb0b0a})), and }Ker(\partial)=\mathbb{C}1_A,$ there exists $\chi\in \mathbb{C}$ such that $(b_0)^2 a=\chi 1_A+c_0a+c_1 b_0a$. Since $(b+A\partial(A))_0(a\cdot (b+A\partial(A))=a\cdot((b+A\partial(A))_0(b+A\partial(A)))+(b_0a)\cdot (b+A\partial(A)),$ and $a\cdot b=\beta b+\gamma_0\partial(a)+\gamma_1\partial(b_0a)$ we have 
$$(b_0a)\cdot (b+A\partial(A))=0+A\partial(A),\text{ and }(b_0a)\cdot b\in A\partial(A).$$ Using the fact that $B/A\partial(A)$ is a Lie $A$-algebroid, we have 
$$a*(b_0a)=a*((b+A\partial(A))_0a)=(a\cdot (b+A\partial(A))_0a=(\beta b+A\partial(A))_0a=\beta b_0a.$$
Because $(b_0a)\cdot b\in A\partial(A)$, we then obtain that $$(b_0a)*(b_0a)=(b_0a)*((b+A\partial(A))_0a)=((b_0a)\cdot (b+A\partial(A)))_0a=0.$$ 
Since $(a\cdot b)_1b=a*(b_1b)-b_0b_0a$, we have
\begin{eqnarray*}
a*a&=&a*(b_1b)\\
&=&(a\cdot b)_1 b+(b_0)^2a\\
&=&(\beta b+\gamma_0\partial(a)+\gamma_1\partial(b_0a) )_1b+(b_0)^2a\\
&=&\beta a+\gamma_0b_0a+\gamma_1(b_0)^2a+(b_0)^2a\\
&=&\beta a +\gamma_0 b_0a+(\gamma_1+1)(b_0)^2a\\
&=&\beta a +\gamma_0 b_0a+(\gamma_1+1)(\chi 1_A+c_0 a+c_1 b_0a )\\
&=&(\gamma_1+1)\chi 1_A+(\beta +(\gamma_1+1)c_0)a+(\gamma_0+(\gamma_1+1)c_1)b_0a.
\end{eqnarray*}
Also, since $b_0(a*a)=2a*(b_0a)=2\beta b_0a$, we can conclude that 
\begin{eqnarray*}
&&2\beta b_0a\\
&&=b_0(a*a)\\
&&=(\beta +(\gamma_1+1)c_0)b_0a+(\gamma_0+(\gamma_1+1)c_1)(b_0)^2a\\
&&=(\beta +(\gamma_1+1)c_0)b_0a+(\gamma_0+(\gamma_1+1)c_1)(\chi 1_A+c_0a+c_1 b_0a )\\
&&=(\gamma_0+(\gamma_1+1)c_1)\chi 1_A+(\gamma_0+(\gamma_1+1)c_1 )c_0a+(\beta+(\gamma_1+1)c_0+(\gamma_0+(\gamma_1+1)c_1)c_1 )b_0a.
\end{eqnarray*}
Using the fact that $\{1_A, a, b_0a\}$ is a basis for $A$, we have 
\begin{eqnarray*}
&&(\gamma_0+(\gamma_1+1)c_1)\chi =0,\\
&&(\gamma_0+(\gamma_1+1)c_1 )c_0=0,\\
&&(\gamma_1+1)c_0+(\gamma_0+(\gamma_1+1)c_1)c_1=\beta.
\end{eqnarray*}
\end{proof}
\begin{thm}\label{3nilcase} 
Let $A$ be a unital commutative associative algebra with the identity $1_A$. Let $B$ be a vertex $A$-algebroid that has properties as in Lemma \ref{dim3case}.

Assume that $b_0\partial(b_0a)=0$. Then 
\begin{eqnarray*}
&&\beta=0,~ a*(b_0a)=0,~ a*a=(\gamma_1+1)\chi 1_A+\gamma_0 b_0a,~a\cdot b=\gamma_0\partial(a)+\gamma_1\partial(b_0a),\\
&&a\cdot \partial(a)=\frac{1}{2}\gamma_0\partial(b_0a),\text{ and }\gamma_0\chi=0.
\end{eqnarray*}
\noindent (i) When $\gamma_0=0$, we have \begin{eqnarray*}
&&(b_0a)\cdot b=0,~\chi=0,~a*a=0,~a\cdot\partial(a)=0,~a\cdot b=\gamma_1\partial(b_0a),\\
&&A\cong \mathbb{C}[x,y]/(x^2,y^2,xy).
\end{eqnarray*}
The vector space $(a, b_0a)$ is the unique maximal ideal of $A$. In addition, for $u\in (a, b_0a)$, $w\in B/A\partial(A)$, $w_0u\in (a, b_0a)$ and $u\cdot w=0$.
    
\noindent (ii) When $\gamma_0\neq 0$, we have \begin{eqnarray*}
&&a*a=\gamma_0 b_0a,~a*a*a=0,~(b_0a)\cdot b=\frac{3}{4}\gamma_0\partial(b_0a),~A\cong \mathbb{C}[x]/(x^3).
\end{eqnarray*} 
The vector space $(a )$ is the unique maximal ideal of $A$. For $u\in (a)$, $w\in B/A\partial(A)$, $w_0u\in (a)$ and $u\cdot w=0$.
\end{thm}
\begin{proof} By setting $c_0$, $c_1$ in Lemma \ref{dim3case} to be $0$, we have \begin{eqnarray*}
&&\beta=0,~ a*(b_0a)=0,~ a*a=(\gamma_1+1)\chi 1_A+\gamma_0 b_0a,~a\cdot b=\gamma_0\partial(a)+\gamma_1\partial(b_0a),\\
&&a\cdot \partial(a)=\frac{1}{2}\gamma_0\partial(b_0a),\text{ and }\gamma_0\chi=0.
\end{eqnarray*}

Now, we first consider the case when $\gamma_0=0$. Recall that for $u,v\in B, a'\in A$, we have $u_0(a'\cdot v)-a'\cdot(u_0v)=(u_0a')\cdot v.$ When we set $u=v=b$ and $a'=a$, we have 
$$b_0(a\cdot b)-a\cdot (b_0b)=(b_0a)\cdot b.$$ Since $a\cdot b_0b=\frac{1}{2}a\cdot\partial(a)=0$ and $a\cdot b=\gamma_1\partial(b_0a)$, we have 
\begin{equation}(b_0a)\cdot b=b_0(\gamma_1\partial(b_0a))=\gamma_1\partial(b_0b_0a)=0.\end{equation} Recall that for $a',a''\in A$, $u\in B$, we have $a'\cdot(a''\cdot u)-(a'*a'')\cdot u=(u_0a')\cdot \partial(a'')+(u_0a'')\cdot\partial(a').$ Now, if we set $a'=a$, $a''=b_0a$, $u=b$, then
$$a\cdot(b_0a\cdot b)-(a*(b_0a))\cdot b=(b_0a)\cdot \partial(b_0a)+(b_0b_0a)\cdot \partial(a). $$ Because $b_0a\cdot \partial(b_0a)=0$, $(b_0a)\cdot b=0$ and $a*(b_0a)=0$, we have $(b_0b_0a)\cdot \partial(a)=0$. 
Also, because $b_0b_0a=\chi 1_A+c_0a+c_1b_0a$, and $c_0=c_1=0$, we have $b_0b_0a=\chi 1_A$. Since $(b_0b_0a)\cdot \partial(a)=0$, we have $\chi \partial(a)=0$. Therefore, \begin{equation}\chi =0,\text{ and }a*a=0.\end{equation} Hence, $A\cong\mathbb{C}[x,y]/(x^2,y^2,xy)$. In addition, the vector space $(a, b_0a)$ is the unique maximal ideal of $A$. For $u\in (a, b_0a)$, $w\in B/A\partial(A)$, $w_0u\in (a, b_0a)$ and $u\cdot w=0$ This completes the case when $\gamma_0=0$.

Next, we consider the case when $\gamma_0\neq 0$. Because $\gamma_0\chi=0$ and $\gamma_0\neq 0$, we have $\chi=0$. Since $\chi=0$ and $a*a=(\gamma_1+1)\chi 1_A+\gamma_0b_0a$, we have \begin{equation}a*a=\gamma_0b_0a.\end{equation} Since $a*(b_0a)=0$, we can conclude that \begin{equation}a*(a*a)=a*(\gamma_0b_0a)=0.\end{equation}
Recall that for $u,v\in B$, $a'\in A$, we have $u_0(a'\cdot v)-a'\cdot (u_0v)=(u_0a')\cdot v$. When we set $u,v=b$ and $a'=a$, we have 
$$b_0(a\cdot b)-a\cdot b_0b=(b_0a)\cdot b.$$ 
Since $a\cdot b=\gamma_0\partial(a)+\gamma_1\partial(b_0a)$ and $b_0b=\frac{1}{2}\partial(a)$, we then have that $$b_0(\gamma_0\partial(a)+\gamma_1\partial(b_0a))-a\cdot \frac{1}{2}\partial(a)=(b_0a)\cdot b.$$ Equivalently,
$\gamma_0\partial(b_0a)+\gamma_1\partial(b_0b_0a)-\frac{1}{2}a\cdot\partial(a)=(b_0a)\cdot b$. Since $\partial(b_0b_0a)=0$ and $a\cdot\partial(a)=\frac{1}{2}\gamma_0\partial(b_0a)$, we have 
\begin{equation}(b_0a)\cdot b=\gamma_0\partial(b_0a)-\frac{1}{4}\gamma_0\partial(b_0a)=\frac{3}{4}\gamma_0\partial(b_0a).\end{equation}
Hence, $A\cong \mathbb{C}[x]/(x^3)$. Moreover, $(a )$ is the unique maximal ideal of $A$. For $u\in (a)$, $w\in B/A\partial(A)$, $w_0u\in (a)$ and $u\cdot w=0$. This completes the case when $\gamma_0\neq 0$.
\end{proof}


\begin{thm}\label{3solvablecase1} 
Let $A$ be a unital commutative associative algebra with the identity $1_A$. Let $B$ be a vertex $A$-algebroid that has properties as in Lemma \ref{dim3case}.

Assume that 
$$b_0\partial(b_0a)=\partial(a)-\left(\frac{\alpha+1}{\sqrt{\alpha}}\right)i \partial(b_0a)~(\alpha\notin\{0,1\}).$$ Then the following statments hold:
\begin{eqnarray*}
&&(\gamma_0+(\gamma_1+1)(-\left(\frac{\alpha+1}{\sqrt{\alpha}}\right)i  ))=0,~\beta=\gamma_1+1,~a*(b_0a)=(\gamma_1+1)b_0a,\\
&&a\cdot \partial(a)=(\gamma_1+1)\partial(a),\\
&&(b_0a)\cdot b=\gamma_1\partial(a)+\frac{\alpha+1}{\sqrt{\alpha}}i\partial(b_0a),\\
&&a\cdot b=(\gamma_1+1) b+(\gamma_1+1)\left(\frac{\alpha+1}{\sqrt{\alpha}}\right) i\partial(a)+\gamma_1\partial(b_0a).
\end{eqnarray*}

\noindent (i) If $\gamma_1+1\neq 0$ then 
\begin{eqnarray*}
&&\chi=-\gamma_1-1,\\
&&b_0b_0a=(-\gamma_1-1 )1_A+a-\left(\frac{\alpha+1}{\sqrt{\alpha}}\right)i b_0a,\\
&&a*a=-(\gamma_1+1)^2 1_A+2(\gamma_1+1)a,\\
&&A\cong \mathbb{C}[x,y]/( (x-(\gamma_1+1))^2, (x-(\gamma_1+1))y, y^2).
\end{eqnarray*}
The vector space $(a-(\gamma_1+1) ) 1_A )$ is the unique maximal ideal of $A$. In addition, for $u\in (a-( \gamma_1+1_A) 1_A)$, $w\in B/A\partial(A)$, $w_0u\in (a-(\gamma_1+1 ) 1_A)$ and $u\cdot w=0$.

\vspace{0.2cm}

\noindent (ii) If $\gamma_1+1=0$, then 
\begin{eqnarray*}
&&\chi=0,\\ 
&&b_0b_0a=a-\left(\frac{\alpha+1}{\sqrt{\alpha}}\right)i b_0a,\\
&&a*a=0,\\
&&a\cdot b=-\partial(b_0a),\\
&&A\cong \mathbb{C}[x,y]/( x^2, xy, y^2).
\end{eqnarray*}
The vector space $(a, b_0a )$ is the unique maximal ideal of $A$. Also, for $u\in (a,b_0a)$, $w\in B/A\partial(A)$, $w_0u\in (a, b_0a)$ and $u\cdot w=0$.
\end{thm}
\begin{proof} When we set $c_0=1$, $c_1=-\left(\frac{\alpha+1}{\sqrt{\alpha}}\right)i $ in Lemma \ref{dim3case}, we have 
\begin{eqnarray*}
&&(\gamma_0+(\gamma_1+1)(-\left(\frac{\alpha+1}{\sqrt{\alpha}}\right)i  ))\chi=0,\\
&&(\gamma_0+(\gamma_1+1)(-\left(\frac{\alpha+1}{\sqrt{\alpha}}\right)i  ))=0,\\
&&\beta=\gamma_1+1+(\gamma_0+(\gamma_1+1)( -\left(\frac{\alpha+1}{\sqrt{\alpha}}\right)i ))(-\left(\frac{\alpha+1}{\sqrt{\alpha}}\right)i ).
\end{eqnarray*}
These imply that 
\begin{eqnarray*}
&&\beta=\gamma_1+1,\\
&&a*(b_0a)=(\gamma_1+1)b_0a,\\
&&a*a=(\gamma_1+1)\chi 1_A+2(\gamma_1+1)a,\\
&&a\cdot \partial(a)=(\gamma_1+1)\partial(a).
\end{eqnarray*}
Moreover, $A\cong \mathbb{C}[x,y]/( x^2-2(\gamma_1+1)x-(\gamma_1+1)\chi, xy-(\gamma_1+1)y, y^2)$.

Since $\beta=\gamma_1+1$, $\gamma_0=(\gamma_1+1)\left(\frac{\alpha+1}{\sqrt{\alpha}}\right) i$, we have that
$$a\cdot b=(\gamma_1+1) b+(\gamma_1+1)\left(\frac{\alpha+1}{\sqrt{\alpha}}\right) i\partial(a)+\gamma_1\partial(b_0a). $$ 
Because
$b_0(a\cdot b)-a\cdot (b_0b)=(b_0a)\cdot b$, and $a\cdot \partial(a)=(\gamma_1+1)\partial(a)$, we have
\begin{eqnarray*}
&&(b_0a)\cdot b\\
&&=b_0((\gamma_1+1) b+(\gamma_1+1)\left(\frac{\alpha+1}{\sqrt{\alpha}}\right) i\partial(a)+\gamma_1\partial(b_0a) )-\frac{1}{2}a\cdot \partial (a)\\
&&=b_0((\gamma_1+1) b+(\gamma_1+1)\left(\frac{\alpha+1}{\sqrt{\alpha}}\right) i\partial(a)+\gamma_1\partial(b_0a) )-\frac{1}{2}(\gamma_1+1)\partial(a)\\
&&=\gamma_1\partial(a)+\frac{\alpha+1}{\sqrt{\alpha}}i\partial(b_0a).
\end{eqnarray*} 
Since $c_0=1$, and $c_1=-\frac{\alpha+1}{\sqrt{\alpha}}i$, we have 
$b_0b_0a=\chi 1_A+a-\frac{\alpha+1}{\sqrt{\alpha}}i b_0a$. Because $$a\cdot (b_0a\cdot b)-(a*b_0a)\cdot b=(b_0a)\cdot \partial(b_0a)+(b_0b_0a)\cdot \partial(a),$$
$a*(b_0a)=(\gamma_1+1)b_0a$, and $(b_0a)\cdot \partial(b_0a)=0$, we have 
$$a\cdot (\gamma_1\partial(a)+\left(\frac{\alpha+1}{\sqrt{\alpha}}\right) i\partial(b_0a))-((\gamma_1+1)b_0a)\cdot b=(\chi 1_A+a-\frac{(\alpha+1)}{\sqrt{\alpha}}i b_0a)\cdot \partial(a).$$
This implies that 
\begin{eqnarray*}
&&\chi \partial(a)\\
&&=(\gamma_1-1)a\cdot \partial(a)+\left(\frac{\alpha+1}{\sqrt{\alpha}}\right) i (a\cdot \partial(b_0a)+(b_0a)\cdot \partial(a))-(\gamma_1+1)(b_0a)\cdot b\\
&&=(\gamma_1-1)(\gamma_1+1) \partial(a)+\left(\frac{\alpha+1}{\sqrt{\alpha}}\right) i \partial(a*(b_0a))-(\gamma_1+1)(b_0a)\cdot b\\
&&=(\gamma_1^2-1)\partial(a)+\left(\frac{\alpha+1}{\sqrt{\alpha}}\right) i \partial((\gamma_1+1)b_0a)-(\gamma_1+1)(b_0a)\cdot b\\
\end{eqnarray*}
The above statement is equivalent to the following:
$$(\gamma_1^2-\chi-1)\partial(a)+\left(\frac{\alpha+1 }{\sqrt{\alpha}}\right) i(\gamma_1+1)\partial(b_0a)=(\gamma_1+1)(b_0a)\cdot b.$$

If $(\gamma_1+1)\neq 0$ then $(b_0a)\cdot b=\frac{\gamma_1^2-1-\chi}{\gamma_1+1}\partial(a)+\left(\frac{\alpha+1}{\sqrt{\alpha}}\right)i\partial(b_0a).$
Since 
\begin{eqnarray*}
&&(b_0a)\cdot b=\gamma_1\partial(a)+\left(\frac{\alpha+1}{\sqrt{\alpha}}\right)i\partial(b_0a)\text{ and }\\
&&(b_0a)\cdot b=\frac{\gamma_1^2-1-\chi}{\gamma_1+1}\partial(a)+\left(\frac{\alpha+1}{\sqrt{\alpha}}\right)i\partial(b_0a),
\end{eqnarray*} we can conclude that 
$\chi=-\gamma_1-1$ and $$b_0b_0a=(-\gamma_1-1 )1_A+a-\left(\frac{\alpha+1}{\sqrt{\alpha}}\right)i b_0a.$$ In addition, $A\cong \mathbb{C}[x,y]/( (x-(\gamma_1+1))^2, (x-(\gamma_1+1))y, y^2)$. The vector space $(a-(\gamma_1+1) ) 1_A )$ is the unique maximal ideal of $A$. In addition, for $u\in (a-( \gamma_1+1_A) 1_A)$, $w\in B/A\partial(A)$, $w_0u\in (a-(\gamma_1+1 ) 1_A)$ and $u\cdot w=0$.

If $\gamma_1+1=0$, then $\chi \partial(a)=0$. This implies that $\chi=0$ and $b_0b_0a=a-\left(\frac{\alpha+1}{\sqrt{\alpha}}\right)i b_0a.$ In addition, we have $\beta=0$, $a*b_0a=0$, $a*a=0$, $a\cdot\partial(a)=0$, $\gamma_0=0$ and $a\cdot b=-\partial(b_0a)$. The unital commutativ associative algebra $A\cong \mathbb{C}[x,y]/( x^2, xy, y^2)$. The vector space $(a, b_0a )$ is the unique maximal ideal of $A$. Also, for $u\in (a,b_0a)$, $w\in B/A\partial(A)$, $w_0u\in (a, b_0a)$ and $u\cdot w=0$.
\end{proof}
\begin{thm}\label{3solvablecase2} Let $A$ be a unital commutative associative algebra with the identity $1_A$. Let $B$ be a vertex $A$-algebroid that has properties as in Lemma \ref{dim3case}. 

Assume that 
$b_0\partial(b_0a)=\partial(a)+2i \partial(b_0a).$ Then the following statments hold:
\begin{eqnarray*}
&&\gamma_0+(\gamma_1+1)2i=0,~\beta=\gamma_1+1,\\
&&a*(b_0a)=(\gamma_1+1)b_0a,~a*a=-(\gamma_1+1)^21_A+2(\gamma_1+1)a,\\
&&a\cdot \partial(a)=(\gamma_1+1)\partial(a),~(b_0a)\cdot b=\gamma_1\partial(a)-2i\partial(b_0a).
\end{eqnarray*} In addition, $A\cong \mathbb{C}[x,y]/( (x-(\gamma_1+1))^2, xy-(\gamma_1+1)y, y^2).$
The vector space $(a-(\gamma_1+1 ) 1_A, b_0a )$ is the unique maximal ideal of $A$. In addition, for $u\in (a-( \gamma_1+1) 1_A, b_0a)$, $w\in B/A\partial(A)$, $w_0u\in (a-(\gamma_1+1 )1_A)$ and $u\cdot w=0$.
\end{thm}
\begin{proof} In Lemma \ref{dim3case}, when we set $c_0=1$, and $c_1=2i$, the following statements hold
\begin{eqnarray*}
&&\gamma_0=-(\gamma_1+1)2i,~\beta=\gamma_1+1,\\
&&a*(b_0a)=(\gamma_1+1)b_0a,~a*a=(\gamma_1+1)\chi 1_A+2(\gamma_1+1)a,\\
&&a\cdot \partial(a)=(\gamma_1+1)\partial(a),~a\cdot b=(\gamma_1+1)b-(\gamma_1+1)2i\partial(a)+\gamma_1\partial(b_0a).
\end{eqnarray*}
Since $(b_0a)\cdot b=b_0(a\cdot b)-a\cdot (b_0b)$, we have 
\begin{equation}(b_0a)\cdot b=\gamma_1\partial(a)-2i\partial(b_0a).\end{equation}
Because
\begin{eqnarray*}
&&b_0b_0a=\chi 1_A+a+2i b_0a,~a\cdot ((b_0a)\cdot b)-(a*(b_0a))\cdot b=(b_0a)\cdot \partial(b_0a)+(b_0b_0a)\cdot \partial(a),\\
&&a*(b_0a)=(\gamma_1+1)b_0a,~(\gamma_1+1)\partial(b_0a)=\partial(a*b_0a)=a\cdot\partial(b_0a)+(b_0a)\cdot \partial(a),
\end{eqnarray*}
we have 
\begin{eqnarray*}
&&a\cdot (\gamma_1\partial(a)-2i\partial(b_0a))-(\gamma_1+1)(b_0a)\cdot b=(\chi 1_A+a+2i b_0a)\cdot \partial(a).
\end{eqnarray*} 
Equivalently, 
$(\chi+\gamma_1+1)\partial(a)=0.$ Therefore, $\chi=-(\gamma_1+1)$, and $$a*a=-(\gamma_1+1)^21_A+2(\gamma_1+1)a.$$ The rest follows immediately.
\end{proof}
\begin{thm}\label{3solvablecase3} 
Let $A$ be a unital commutative associative algebra with the identity $1_A$. Let $B$ be a vertex $A$-algebroid that has properties as in Lemma \ref{dim3case}.

Assume that 
$b_0\partial(b_0a)= \partial(b_0a).$ Then the following statments hold:
\begin{eqnarray*}
&&\chi=0,~\beta=\gamma_0+\gamma_1+1,\\
&&a*(b_0a)=\beta b_0a,~a*a=\beta a+\beta b_0a,\\
&&a\cdot \partial(a)=\frac{1}{2}(\beta\partial(a)+\frac{1}{2}\beta \partial(b_0a)).
\end{eqnarray*}
\begin{enumerate}
\item If $\gamma_0+\gamma_1+1=0$, then 
\begin{eqnarray*}
&&a\cdot b=(-1-\gamma_1)\partial(a)+\gamma_1\partial(b_0a),~b_0b_0a=b_0a,\\
&&a*(b_0a)=0,~a*a=0,~a\cdot \partial(a)=0,~(b_0a)\cdot b=-\partial(b_0a),\\
&&A\cong \mathbb{C}[x,y]/( x^2, y^2,xy).
\end{eqnarray*}
The vector space $(a, b_0a )$ is the unique maximal ideal of $A$. For $u\in (a,b_0a)$, $w\in B/A\partial(A)$, $w_0u\in (a,b_0a)$ and $u\cdot w=0$.

\item If $\gamma_0+\gamma_1+1\neq 0$, then 
\begin{eqnarray*}
&&b_0b_0a=b_0a,~ a*a=(\gamma_0+\gamma_1+1)(a+b_0a),\text{ and }\\
&&A\cong \mathbb{C}[x,y]/( x^2-(\gamma_0+\gamma_1+1)(x+y), xy-(\gamma_0+\gamma_1+1)y, y^2).
\end{eqnarray*} 
The vector space $(a-(\gamma_0+\gamma_1+1) 1_A, b_0a )$ is the unique maximal ideal of $A$. In addition, for $u\in (a-(\gamma_0+\gamma_1+1 ) 1_A, b_0a)$, $w\in B/A\partial(A)$, $w_0u\in (a-(\gamma_0+\gamma_1+1 ) 1_A,b_0a)$ and $u\cdot w=0$.
\end{enumerate}

\end{thm}
\begin{proof} If we set $c_0=0$, $c_1=1$ in Lemma \ref{dim3case}, we then have that 
\begin{eqnarray*}
&&(\gamma_0+(\gamma_1+1))\chi=0,~\beta=\gamma_0+\gamma_1+1,\\
&&a*a=(\gamma_1+1)\chi 1_A+(\gamma_0+\gamma_1+1)a+(\gamma_0+\gamma_1+1)b_0a,\\
&&a\cdot \partial(a)=\frac{1}{2}((\gamma_0+\gamma_1+1)\partial(a)+(\gamma_0+\gamma_1+1)\partial(b_0a)).
\end{eqnarray*}
Because $(b_0a)\cdot b=b_0(a\cdot b)-a\cdot (b_0b)$, we have
\begin{eqnarray*}
&&(b_0a)\cdot b\\
&&=b_0(\beta b+\gamma_0\partial(a)+\gamma_1\partial(b_0a))-\frac{1}{2}a\cdot \partial(a)\\
&&=\frac{1}{2}\beta \partial(a)+\gamma_0\partial(b_0a)+\gamma_1\partial(b_0a)-\frac{1}{2}(\frac{1}{2}\beta\partial(a)+\frac{1}{2}\beta\partial(b_0a))\\
&&=\frac{1}{4}\beta\partial(a)+(\gamma_0+\gamma_1-\frac{1}{4}\beta)\partial(b_0a)\\
&&=\frac{1}{4}\beta\partial(a)+(\frac{3}{4}\beta-1)\partial(b_0a). 
\end{eqnarray*}

Now, assume that $\gamma_0+\gamma_1+1=0$. For this case, we have
\begin{eqnarray*}
&&\beta=0,~b_0b_0a=\chi 1_A+b_0a,~a*b_0a=\beta b_0a=0,\\
&&(b_0a)\cdot \partial(b_0a)=0,~a\cdot \partial(a)=0,~(b_0a)\cdot b=-\partial(b_0a).
\end{eqnarray*}
Recall that $ a\cdot(b_0a\cdot b)-(a*b_0a)\cdot b=(b_0a)\cdot \partial(b_0a)+(b_0b_0a)\cdot \partial(a)$. Since
$$a\cdot ((b_0a)\cdot b)-(a*(b_0a))\cdot b=a\cdot (-\partial(b_0a))$$ and 
$(b_0a)\cdot \partial(b_0a)+(b_0b_0a)\cdot \partial(a)=(\chi 1_A+b_0a)\cdot \partial(a)$, we can conclude that 
$$\chi\partial(a)+(b_0a)\cdot \partial(a)=-a\cdot \partial(b_0a).$$ Equivalently, we have $\chi \partial(a)+\partial(a*b_0a)=0$. Since $a*b_0a=0$, we then have that $\chi \partial(a)=0$. This implies that $\chi=0$, and $a*a=0$. The unital commutative associative algebra $A\cong \mathbb{C}[x,y]/( x^2, y^2,xy)$. The vector space $(a, b_0a )$ is the unique maximal ideal of $A$. For $u\in (a,b_0a)$, $w\in B/A\partial(A)$, $w_0u\in (a,b_0a)$ and $u\cdot w=0$.

Next, we assume that $\gamma_0+\gamma_1+1\neq 0$. Then $\chi=0$, $b_0b_0a=b_0a$, and $$a*a=(\gamma_0+\gamma_1+1)(a+b_0a).$$ The unital commutative associative algebra $A\cong \mathbb{C}[x,y]/( x^2-(\gamma_0+\gamma_1+1)(x+y), xy-(\gamma_0+\gamma_1+1)y, y^2)$. The vector space $(a-(\gamma_0+\gamma_1+1) 1_A, b_0a )$ is the unique maximal ideal of $A$. In addition, for $u\in (a-(\gamma_0+\gamma_1+1 ) 1_A, b_0a)$, $w\in B/A\partial(A)$, $w_0u\in (a-(\gamma_0+\gamma_1+1 ) 1_A,b_0a)$ and $u\cdot w=0$.

\end{proof}                                                                                                                  

\section{On vertex algebras associated to non-Lie cyclic left Leibniz algebras}

For this section, we discuss representation theory of vertex algebras $V_B$ associated to vertex $A$-algebroids $B$ when $B$ are non-Lie cyclic left Leibniz algebras, and $2\leq \dim B\leq 3$ from various aspects. Precisely, in subsection \ref{generaltheory}, we review a general construction of a vertex algebra $V_B$ associated to an arbitrary vertex $A$-algebroid $B$ and its modules. In subsection \ref{moduletheory} and subsection \ref{relationtoHeisenberg}, we focus our attention on vertex algebras $V_B$ when $B$ are vertex $A$-algebroids that are non-Lie cyclic left Leibniz algebras, and $2\leq \dim B\leq 3$. Also, we discuss representation theory of vertex algebras $V_B$ associated to vertex $A$-algebroids $B$ from representation theory of Lie algebroid point of view, and describe relations between specific types of vertex algebras $V_B$, and a rank one Heisenberg vertex algebra. 

\begin{prop}\cite{JY}Let $V=\oplus_{n=0}^{\infty}V_{(n)}$ be a $\mathbb{N}$-graded vertex algebra such that $V_{(0)}$ is a finite dimensional unital commutative associative algebra and $\dim V_{(0)}\geq 2$. If $V_{(0)}$ is a local algebra then $V$ is indecomposable. 
\end{prop}

\begin{prop}\label{indecomposableVB}\cite{JY} Let $V=\oplus_{n=0}^{\infty}V_{(n)}$ be a $\mathbb{N}$-graded vertex algebra that satisfies the following properties:

\vspace{0.2cm}

\noindent (a) $2\leq \dim V_{(0)}<\infty$, $1\leq dim V_{(1)}<\infty$, $V$ is generated by $V_{(0)}$ and $V_{(1)}$;

\vspace{0.2cm}

\noindent (b) $V_{(0)}$ is a local algebra.

\vspace{0.2cm}
\noindent If $V_{(0)}$ is not a simple module for a Lie $V_{(0)}$-algebroid $V_{(1)}/{V_{(0)}}D(V_{(0)})$, then $V$ is an indecomposable non-simple vertex algebra. 
\end{prop}

\begin{thm}\label{v0v1cyclic} Let $V=\oplus_{n=0}^{\infty}V_{(n)}$ be an $\mathbb{N}$-graded vertex algebra that satisfies the following properties:

\vspace{0.2cm}

\noindent (a) $2\leq \dim V_{(0)}<\infty$, $1\leq dim V_{(1)}<\infty$, $V$ is generated by $V_{(0)}$ and $V_{(1)}$;

\vspace{0.2cm}

\noindent (b) $V_{(1)}$ is a cyclic left Leibniz algebra such that $\dim V_{(1)}$ is either 2 or 3;

\noindent (c) $Ker(D|_{V_{(0)}})=\mathbb{C}{\bf 1}$

\vspace{0.2cm}
\noindent Then $V$ is an indecomposable non-simple $\mathbb{N}$-graded vertex algebra. 
\end{thm}
\begin{proof} Since $V$ is an $\mathbb{N}$-graded vertex algebra, we can conclude immediately that $V_1$ is a vertex $V_0$-algebroid. Since $Ker(D|_{V_{(0)}})=\mathbb{C}{\bf 1}$, we can conclude immediately that $V_{(0)}$ is a local algebra. By Theorem \ref{Bcyclicnilpotent}, Theorem \ref{Bcyclicsolvable}, Theorem \ref{3nilcase}-Theorem \ref{3solvablecase3}, we have that $V_0$ is not a simple module for the Lie $V_{(0)}$-algebroid $V_{(1)}/V_{(0)}D(V_{(0)})$. 
By Proposition \ref{indecomposableVB}, we can conclude that $V$ is an indecomposable non-simple $\mathbb{N}$-graded vertex algebra. 
    
\end{proof}

\subsection{A review on vertex algebras $V_B$ associated with vertex algebroids $B$ and their modules}\label{generaltheory}

\noindent In this subsection, we recall a construction of vertex algebras associated with vertex algebroids in \cite{LiY}. 

\vspace{0.2cm}

\noindent Let $A$ be a commutative associative algebra with identity $1_A$ and let $B$ be a vertex $A$-algebroid. We set $L(A\oplus B)=(A\oplus B)\otimes \mathbb{C}[t,t^{-1}].$ 
Subspaces $L(A)$ and $L(B)$ of $L(A\oplus B)$ are defined in the obvious way. 

\vspace{0.2cm}

\noindent We set $\hat{\partial}=\partial\otimes 1+1\otimes \frac{d}{dt}:L(A)\rightarrow L(A\oplus B).$ We define 
$deg(a\otimes t^n)=-n-1$, $deg(b\otimes t^n)=-n$ for $a\in A, ~b\in B,~n\in\mathbb{Z}$.
Then $L(A\oplus B)$ becomes a $\mathbb{Z}$-graded vector space: 
$$L(A\oplus B)=\oplus_{n\in\mathbb{Z}}L(A\oplus B)_{(n)}$$ where 
$L(A\oplus B)_{(n)}=A\otimes \mathbb{C}t^{-n-1}+B\otimes \mathbb{C}t^{-n}$. Clearly, the subspaces $L(A)$ and $L(B)$ are $\mathbb{Z}$-graded vector spaces as well. In addition, for $n\in \mathbb{N}$, $L(A)_{(n)}=A\otimes  \mathbb{C}t^{-n-1}.$ 
The linear map $\hat{\partial}:L(A)\rightarrow L(A\oplus B)$ is of degree 1. We define a bilinear product $[\cdot,\cdot]$ on $L(A\oplus B)$ as follow:
\begin{eqnarray}
&&[a\otimes t^m,a'\otimes t^n]=0,\label{aa'}\\
&&[a\otimes t^m, b\otimes t^n]=a_0b\otimes t^{m+n},\\
&&[b\otimes t^n,a\otimes t^m]=b_0a\otimes t^{m+n},\\
&&[b\otimes t^m,b'\otimes t^n]=b_0b'\otimes t^{m+n}+m(b_1b')\otimes t^{m+n-1}\label{bb'}
\end{eqnarray}
for $a,a'\in A$, $b,b'\in B$, $m,n\in\mathbb{Z}$. For convenience, we set 
$$\mathcal{L}:=L(A\oplus B)/\hat{\partial}L(A).$$ It was shown in \cite{LiY} that $\mathcal{L}=\oplus_{n\in\mathbb{Z}}\mathcal{L}_{(n)}$ is a $\mathbb{Z}$-graded Lie algebra. Here, 
$$\mathcal{L}_{(n)}=L(A\oplus B)_{(n)}/\hat{\partial}(L(A)_{(n-1)})=(A\otimes \mathbb{C}t^{-n-1}+B\otimes \mathbb{C}t^{-n})/\hat{\partial}(A\otimes\mathbb{C}t^{-n}).$$ In particular, $\mathcal{L}_{(0)}=A\otimes \mathbb{C}t^{-1}+B/\partial A$. 

\vspace{0.2cm}

\noindent Let $\rho:L(A\oplus B)\rightarrow\mathcal{L}$ be a natural linear map defined by $$\rho( u\otimes t^n)=u\otimes t^n+\hat{\partial}L(A).$$ For $u\in A\oplus B$, $n\in\mathbb{Z}$, we set $u(n)=\rho(u\otimes t^n)$ and $u(z)=\sum_{n\in\mathbb{Z}}u(n)z^{-n-1}$. Let $W$ be a $\mathcal{L}$-module. We use $u_W(n)$ or sometimes just $u(n)$ for the corresponding operator on $W$ and we write $u_W(z)=\sum_{n\in\mathbb{Z}}u(n)z^{-n-1}\in (\End W)[[z,z^{-1}]]$. The commutator relations in terms of generating functions are the following:

\begin{eqnarray}
&&[a(z_1),a'(z_2)]=0,\label{vaaa'}\\
&&[a(z_1), b(z_2)]=z_2^{-1}\delta\left(\frac{z_1}{z_2}\right)(a_0b)(z_2),\label{vaab}\\
&&[b(z_1),b'(z_2)]=z_2^{-1}\delta\left(\frac{z_1}{z_2}\right)(b_0b')(z_2)+(b_1b')(z_2)\frac{\partial}{\partial z_2}z_2^{-1}\delta\left(\frac{z_1}{z_2}\right)\label{vabb'}
\end{eqnarray}
for $a,a'\in A$, $b,b'\in B$. 
\vspace{0.2cm}

\noindent Next, we define $\mathcal{L}^{\geq 0}=\rho((A\oplus B)\otimes \mathbb{C}[t])\subset \mathcal{L}$, and $\mathcal{L}^{<0}=\rho((A\oplus B)\otimes t^{-1}\mathbb{C}[t^{-1}])\subset \mathcal{L}$. As a vector space, $\mathcal{L}=\mathcal{L}^{\geq 0}\oplus \mathcal{L}^{<0}$. The subspaces $\mathcal{L}^{\geq 0}$ and $\mathcal{L}^{<0}$ are graded sub-algebras of $\mathcal{L}$.

\vspace{0.2cm}

\noindent We now consider $\mathbb{C}$ as the trivial $\mathcal{L}^{\geq 0}$-module and form the following induced module $$V_{\mathcal{L}}=U(\mathcal{L})\otimes_{U(\mathcal{L}^{\geq 0})}\mathbb{C}.$$ 
In view of the Poincare-Birkhoff-Witt theorem, we have $V_{\mathcal{L}}=U(\mathcal{L}^{<0})$ as a vector space. We set ${\bf 1}=1\in V_{\mathcal{L}}$. We may consider $A\oplus B$ as a subspace: $$A\oplus B\rightarrow V_{\mathcal{L}}, \ \  a+b\mapsto a(-1){\bf 1}+b(-1){\bf 1}.$$ 
We assign $deg~\mathbb{C}=0$. Then $V_{\mathcal{L}}=\oplus_{n\in\mathbb{N}}(V_{\mathcal{L}})_{(n)}$ is a restricted $\mathbb{N}$-graded $\mathcal{L}$-module. 
\begin{prop} \cite{FKRW, MeP} 

\vspace{0.2cm}

\noindent There exists a unique vertex algebra structure on $V_{\mathcal{L}}$ with $Y(u,x)=u(x)$ for $u\in A\oplus B$. In fact, the vertex algebra $V_{\mathcal{L}}$ is a $\mathbb{N}$-graded vertex algebra and it is generated by $A\oplus B$. Furthermore, any restricted $\mathcal{L}$-module $W$ is naturally a $V_{\mathcal{L}}$-module with $Y_W(u,x)=u_W(x)$ for $u\in A\oplus B$. Conversely, any $V_{\mathcal{L}}$-module $W$ is naturally a restricted $\mathcal{L}$-module with $u_W(x)=Y_W(u,x)$ for $u\in A\oplus B$.
\end{prop}

\vspace{0.2cm} 

\noindent Now, we set 
\begin{eqnarray*}
&&E_0=Span\{1_A-{\bf 1},a(-1)a'-a*a'~|~a,a'\in A\}\subset (V_{\mathcal{L}})_{(0)},\\
&&E_1=Span\{a(-1)b-a\cdot b~|~a\in A,b\in\mathfrak{L}\}\subset (V_{\mathcal{L}})_{(1)},\\
&&E=E_0\oplus E_1.
\end{eqnarray*}
We define 
$$I_{B}=U(\mathcal{L})\mathbb{C}[D]E.$$ The vector space $I_{B}$ is an $\mathcal{L}$-submodule of $V_{\mathcal{L}}$. We set $$V_{B}=V_{\mathcal{L}}/I_{B}.$$

\begin{prop}\cite{GMS, LiY} \ \ 

\vspace{0.2cm}

\noindent (i) $V_{B}$ is a $\mathbb{N}$-graded vertex algebra such that $(V_{B})_{(0)}=A$ and $(V_{B})_{(1)}=B$ (under the linear map $v\mapsto v(-1){\bf 1}$) and $V_{B}$ as a vertex algebra is generated by $A\oplus B$. Furthermore, for any $n\geq 1$,
\begin{eqnarray*}
&&(V_{B})_{(n)}\\
&&=span\{b_1(-n_1).....b_k(-n_k){\bf 1}~|~b_i\in B,n_1\geq...\geq n_k\geq 1, n_1+...+n_k=n\}.
\end{eqnarray*}

\vspace{0.2cm}

\noindent (ii) A $V_{B}$-module $W$ is a restricted module for the Lie algebra $\mathcal{L}$ with $v(n)$ acting as $v_n$ for $v\in A\oplus B$, $n\in\mathbb{Z}$. Furthermore, the set of $V_{B}$-submodules is precisely the set of $\mathcal{L}$-submodules.
\end{prop}

\begin{prop}\label{prop30}\cite{LiY} 
Let $W=\oplus_{n\in\mathbb{N}}W_{(n)}$ be a $\mathbb{N}$-graded $V_{B}$-module with $W_{(0)}\neq\{ 0\}$. Then $W_{(0)}$ is an $A$-module with $a\cdot w=a_{-1}w$ for $a\in A$, $w\in W_{(0)}$ and $W_{(0)}$ is a module for the Lie algebra $B/A\partial(A)$ with $b\cdot w=b_0w$ for $b\in B$, $w\in W_{(0)}$. Furthermore, $W_{(0)}$ equipped with these module structures is a module for Lie $A$-algebroid $B/A\partial A$. If $W$ is graded simple, then $W_{(0)}$ is a simple module for Lie $A$-algebroid $B/A\partial A$.

\end{prop}

\vspace{0.2cm} 

\noindent Now, we set $\mathcal{L}_{\pm }=\oplus_{n\geq 1}\mathcal{L}_{(\pm n)}$ and $\mathcal{L}_{\leq 0}=\mathcal{L}_{-}\oplus \mathcal{L}_{(0)}$. Let $U$ be a module for the Lie algebra $\mathcal{L}_{(0)}(=A\oplus B/\partial (A))$. Then $U$ is an $\mathcal{L}_{(\leq 0)}$-module under the following actions:
$$a(n-1)\cdot u=\delta_{n,0}a\cdot u,\ \ b(n)\cdot u=\delta_{n,0}b\cdot u\text{ for }a\in A,b\in B, n\geq 0.$$ 
Next, we form the induced $\mathcal{L}$-module $M(U)=\Ind_{\mathcal{L}_{(\leq 0)}}^{\mathcal{L}}U$. Endow $U$ with degree 0, making $M(U)$ a $\mathbb{N}$-graded $\mathcal{L}$-module. In fact, $M(U)$ is a $V_{\mathcal{L}}$-module. 
We set $$W(U)=span\{v_nu~|~v\in E, ~n\in\mathbb{Z}, ~u\in U\}\subset M(U),$$ and
$$M_{B}(U)=M(U)/U(\mathcal{L})W(U).$$
\begin{prop}\label{simplemodulerelations}\cite{LiY} \ \ 

\vspace{0.2cm} 

\noindent (i) Let $U$ be a module for the Lie algebra $\mathcal{L}_{(0)}$. Then $M_{B}(U)$ is a $V_{B}$-module. If $U$ is a module for the Lie $A$-algebroid $B/A\partial A$ then $(M_{B}(U))_{(0)}=U$.

\vspace{0.2cm} 

\noindent (ii) Let $U$ be a module for the Lie $A$-algebroid $B/A\partial A$. Then there exists a unique maximal graded $U(\mathcal{L})$-submodule $J(U)$  of $M(U)$ with the property that $J(U)\cap U=0$. Moreover, $L(U)=M(U)/J(U)$ is a $\mathbb{N}$-graded $V_{B}$-module such that $L(U)_{(0)}=U$ as a module for the Lie $A$-algebroid $B/A\partial A$. If $U$ is a simple $B/A\partial A$, $L(U)$ is a $\mathbb{N}$-graded simple $V_{B}$-module.

\vspace{0.2cm}

\noindent (iii) Let $W=\coprod_{n\in\mathbb{N}}W_{(n)}$ be an $\mathbb{N}$-graded simple $V_B$-module with $W_{(0)}\neq 0$. Then $W\cong L(W_{(0)})$. 

\vspace{0.2cm}

\noindent (iv) For any complete set $H$ of representatives of equivalence classes of simple modules for the Lie $A$-algebroid $B/A\partial A$, $\{L(U)~|~U\in H\}$ is a complete set of representatives of equivalence classes of $\mathbb{N}$-graded simple $V_B$-modules.
\end{prop}


\subsection{On representation theory of vertex algebra $V_B$ when $B$ is a cyclic Leibniz algebra}\label{moduletheory}

The following theorem is an immediate consequence of Theorem \ref{v0v1cyclic}.
\begin{thm} Let $B$ be a vertex $A$-algebroid as in Theorem \ref{Bcyclicnilpotent}-Theorem \ref{Bcyclicsolvable}, Theorem \ref{3nilcase}- Theorem \ref{3solvablecase3}. Then $V_B$ is an indecomposable non-simple vertex algebra. 
\end{thm}

Now, we will classify irreducible modules of the vertex algebra $V_B$ when $B$ are vertex $A$-algebroids as in Theorem \ref{Bcyclicnilpotent}-Theorem \ref{Bcyclicsolvable}, Theorem \ref{3nilcase}- Theorem \ref{3solvablecase3}. By Proposition \ref{simplemodulerelations}, it is enough to classify simple modules for the Lie $A$-algebroids $B/A\partial(A)$. 

The following lemma is a consequence of the definitions of module of a Lie algebroid, decomposable module and non-irreducible module. 
\begin{lem}\label{indecomposableirreducible} Let $U$ be a Lie $A$-algebroid $B/A\partial(A)$. 

 \noindent (i) If $U$ is a decomposable Lie $A$-algbroid $B/A\partial(A)$-module then $U$ is a decomposable module for the Lie algebra $B/A\partial(A)$ and $U$ is a decomposable $A$-module. Equivalently, if $U$ is either an indecomposable module for the Lie algebra $B/A\partial(A)$ or an indecomposable $A$-module, then $U$ is an indecomposable Lie $A$-algbroid $B/A\partial(A)$-module.

\noindent (ii) If $U$ is not an irreducible Lie $A$-algbroid $B/A\partial(A)$-module then $U$ is not an irreducible module for the Lie algebra $B/A\partial(A)$ and $U$ is not an irreducible $A$-module. Likewise, if $U$ is either an irreducible module for the Lie algebra $B/A\partial(A)$ or an irreducible $A$-module, then $U$ is an irreducible Lie $A$-algbroid $B/A\partial(A)$-module.
\end{lem}

\begin{lem}\label{solvabledimension} Let $(\mathfrak{a}, *)$ be a unital commutative associative algebra with the identity $1_{\mathfrak{a}}$. Let $\mathfrak{g}$ be a Lie $\mathfrak{a}$-algebroid such that $\mathfrak{g}$ is solvable. Assume that 

\noindent (i) $\mathfrak{a}$ is a local algebra and $\mathfrak{a}=\mathbb{C}1_{\mathfrak{a}}\oplus \mathfrak{a}'$. Here $\mathfrak{a}'$ is a unique maximal ideal of $\mathfrak{a}$;

\noindent (ii) for $u\in \mathfrak{a}'$, $w\in \mathfrak{g}$, $w_0u\in \mathfrak{a}'$, and $u\cdot w=0$. 

\noindent Then every finite dimensional irreducible Lie $\mathfrak{a}$-algebroid $\mathfrak{g}$-module $U\neq 0$ is one dimensional. In addition, $\mathfrak{a}'$ acts as zero on $U$.
\end{lem}
\begin{proof} Let $U\neq 0$ be a finite dimensional irreducible Lie $\mathfrak{a}$-algebroid $\mathfrak{g}$-module. Then $U$ is a module for the solvable Lie algebra $\mathfrak{g}$. In addition, $U$ contains a common eigenvector $v$ for all endomorphisms in $\mathfrak{g}$, and $K=\mathbb{C}v$ is a $\mathfrak{g}$-submodule of $U$. For $g\in \mathfrak{g}$, there exists $k_g\in\mathbb{C}$ such that $g\cdot v=k_gv$.

We set $\mathfrak{a}'\cdot v=Span\{\alpha'\cdot v~|~\alpha'\in \mathfrak{a}'\}\subset U$. Clearly, $\mathfrak{a}'\cdot v$ is a $\mathfrak{a}$-module.

\begin{itemize}
\item Case 1: $\mathfrak{g}$ acts trivially on $K$. 

Suppose that $\mathfrak{a}'\cdot v\neq \{0\}$. Notice that for $g\in\mathfrak{g}$, $\alpha\in \mathfrak{a}'$, we have 
$g\cdot (\alpha\cdot v)=\alpha \cdot (g\cdot v)+(g\cdot \alpha)\cdot v=(g\cdot \alpha)\cdot v\in \mathfrak{a}'\cdot v$. Hence, $\mathfrak{a}'\cdot v$ is a $\mathfrak{g}$-module. Since 
\begin{eqnarray*}
&&g\cdot (a\cdot (\alpha\cdot v))-a\cdot (g\cdot (\alpha\cdot v))\\
&&=(g\cdot a)\cdot (\alpha\cdot v)
\end{eqnarray*}
and 
\begin{eqnarray*}
    &&a\cdot (g\cdot (\alpha v))=(a\cdot g)\cdot (\alpha\cdot v)
\end{eqnarray*}
for $g\in\mathfrak{g}, a\in\mathfrak{a},\alpha\in\mathfrak{a}'$, we can conclude that $\mathfrak{a}'\cdot v$ is a Lie $\mathfrak{a}$-algebroid $\mathfrak{g}$-module. Since $U$ is irreducible, $\mathfrak{a}'\cdot v\neq \{0\}$, we can conclude that $U=\mathfrak{a}'\cdot v$. This implies that $v\in\mathfrak{a}'\cdot v$. So, there exists $\lambda\in \mathfrak{a'}$ such that $v=\lambda v$. Equivalently, $(1_{\mathfrak{a}}-\lambda)v=0$. We know that $\lambda$ is not invertible since $\lambda\in \mathfrak{a'}$. Because $\mathfrak{a}$ is local, we can conclude that $1_{\mathfrak{a}}-\lambda$ is invertible. Since $1_{\mathfrak{a}}-\lambda$ is invertible and $(1_{\mathfrak{a}}-\lambda)v=0$, we have $v=0$. This is a contradiction because $v$ is an eigenvector. Hence, $\mathfrak{a}'\cdot v= \{0\}$. Also, $K$ is a module of the Lie $\mathfrak{a}$-algebroid $\mathfrak{g}$. This implies that $U=K$.

\item Case 2: $\mathfrak{g}$ does not trivially on $K$. 

There exists $g'\in \mathfrak{g}$ such that $g'v=k_{g'}v\neq 0$. Notice that $k_{g'} a'\cdot v=a'(g'\cdot v)=(a'\cdot g')\cdot v=0$ for all $a'\in\mathfrak{a}'$. This implies that $K$ is closed under the action of $\mathfrak{a}$, and $a'\cdot v=0$ for all $a'\in\mathfrak{a}'$. In addition, $K$ is a module for the Lie $\mathfrak{a}$-algebroid $\mathfrak{g}$. Since $K$ is a Lie $\mathfrak{a}$-algebroid $\mathfrak{g}$-submodule of $U$ and $U$ is an irreducible Lie $\mathfrak{a}$-algebroid $\mathfrak{g}$-module, we can conclude that $U=K$ is one-dimensional. 
\end{itemize}

\end{proof}

\begin{thm} Let $B$ be a vertex $A$-algebroid as in Theorem \ref{Bcyclicnilpotent}-Theorem \ref{Bcyclicsolvable}, Theorem \ref{3nilcase}-Theorem \ref{3solvablecase3}. If $U$ is a finite dimensional irreducible Lie $A$-algebroid $B/A\partial(A)$-module then there exists $v\in U$ such that $U=\mathbb{C}v$, the unique maximal ideal of $A$ acts as zero on $v$ and $b\cdot v=\lambda v$. Here, $\lambda\in\mathbb{C}$. In addition, $$\{\mathbb{C}v_{\lambda}~|~\lambda\in\mathbb{C},~ b\cdot v_{\lambda}=\lambda v_{\lambda},~\text{ the unique maximal ideal of $A$ acts as zero on $v_\lambda$}\}$$ is the complete set of representatives of equivalence classes of finite dimensional irreducible Lie $A$-algebroid $B/A\partial(A)$-modules, and 
$$\{L(\mathbb{C}v_{\lambda})~|~\lambda\in\mathbb{C},~ b\cdot v_{\lambda}=\lambda v_{\lambda},~\text{ the unique maximal ideal of $A$ acts as zero on $v_\lambda$}\}$$ is the complete set of representatives of equivalence classes of irredudible $\mathbb{N}$-graded $V_B$-modules whose degree zero sub-spaces are finite dimensional. 
\end{thm}
\begin{proof} Let $B$ be a vertex $A$-algebroid as in Theorem \ref{Bcyclicnilpotent}-Theorem \ref{Bcyclicsolvable}, Theorem \ref{3nilcase}-Theorem \ref{3solvablecase3}. Then $A$ is a local algebra with unique maximal ideal $A'$ such that $A=\mathbb{C}1_A\oplus A'$ (as a vector space). $B/A\partial(A)$ is a one dimensional abelian Lie algebra. Moreover, $B/A\partial(A)$ is a Lie $A$-algebroid. For $u\in A'$, $w\in B/A\partial(A)$, $w_0u\in A'$ and $u\cdot w=0+A\partial(A)$. By Lemma \ref{solvabledimension}, we can conclude that every finite dimensional irreducible Lie $A$-algebroid $B/A\partial(A)$-module $U\neq 0$ is one dimensional. In addition, $A'$ acts as zero on $U$. In addition, $$\{\mathbb{C}v_{\lambda}~|~\lambda\in\mathbb{C},~ b\cdot v_{\lambda}=\lambda v_{\lambda},~\text{ the unique maximal ideal of $A$ acts as zero on $v_\lambda$}\}$$ is the complete set of representatives of equivalence classes of finite dimensional irreducible Lie $A$-algebroid $B/A\partial(A)$-modules. By Proposition \ref{simplemodulerelations}, we can conclude further that $$\{L(\mathbb{C}v_{\lambda})~|~\lambda\in\mathbb{C},~ b\cdot v_{\lambda}=\lambda v_{\lambda},~\text{ the unique maximal ideal of $A$ acts as zero on $v_\lambda$}\}$$ is the complete set of representatives of equivalence classes of irredudible $\mathbb{N}$-graded $V_B$-modules whose degree zero sub-spaces are finite dimensional.
\end{proof}

\subsection{On a connection between vertex algebras $V_B$ and a vertex algebra associated with a rank one Heisenberg}\label{relationtoHeisenberg}

In this section, we connect the vertex algebra $V_B$ that we study with a well-known vertex operator algebra associated with a rank one Heisenberg. First, we briefly review construction of vertex algebra associated with rank one Heisenberg algebra. Next, we show that one can construct a Heisenberg vertex operator algebra from certain vertex algebras $V_B$. 


We denote by $\mathfrak{h}$ a one-dimensional abelian Lie algebra spanned by $h$ with a bilinear form $\langle\cdot,\cdot\rangle$ such that $\langle h,h\rangle=1$. We set $$\hat{\mathfrak{h}}=\mathfrak{h}\otimes \mathbb{C}[t,t^{-1}]\oplus\mathbb{C}{\bf{k}}$$ the affinization of $\mathfrak{h}$ with bracket relations: for $a,b\in\mathfrak{h},m, n\in\mathbb{Z}$, 
$$[a(m),b(n)]=m\langle a,b\rangle \delta_{m+n,0}{\bf{k}},~[{\bf{k}},a(m)]=0.$$ Here we define $a(m)=a\otimes t^m$ for $a\in\mathfrak{h}$, $m\in\mathbb{Z}$. We let $$\hat{\mathfrak{h}}^+=\mathfrak{h}\otimes t\mathbb{C}[t]\text{ and }\hat{\mathfrak{h}}^-=\mathfrak{h}\otimes t^{-1}\mathbb{C}[t^{-1}].$$ The vector spaces $\hat{\mathfrak{h}}^+$ and $\hat{\mathfrak{h}}^-$ are abelian subalgebras of $\hat{\mathfrak{h}}$. 

Now, we consider the induced $\hat{\mathfrak{h}}$-module given by $$M(1)={\mathcal{U}}(\hat{\mathfrak{h}})\otimes_{{\mathcal{U}}(\mathbb{C}[t]\otimes \mathfrak{h}\oplus{\mathbb{C}}{\bf k})}\mathbb{C}{\bf 1}\cong S(\hat{\mathfrak{h}}^-).$$ Here $\mathcal{U}(\cdot)$ and $S(\cdot)$ denote the universal enveloping algebra and symmetric algebra, respectively. Also, $\mathfrak{h}\otimes \mathbb{C}[t]$ acts trivially on $\mathbb{C}{\bf 1}$ and $\bf{k}$ acts as multiplication by 1. Then $M(1)$ is a vertex algebra that is often called the vertex operator algebra associated to the rank one Heisenberg algebra or simply the rank one Heisenberg vertex operator algebra, or the one free boson vertex operator algebra. 

Any element of $M(1)$ can be expressed as a linear combination of elements of the form $$h(-k_1).....h(-k_j){\bf 1}\text{ with }k_1\geq...\geq k_j\geq 1,\text{ for }j\in\mathbb{N}.$$ 
It is known that $M(1)$ is simple and has infinitely many non-isomorphic irreducible modules (cf. \cite{FLM2, Gu}). In addition, the indecomposable modules have been completely determined in \cite{Mil}.

\begin{thm} Let $A$ be a unital commutative associative algebra with the identity $1_A$. Let $B$ be a vertex $A$-algebroid that satisfies assumptions in either Theorem \ref{Bcyclicsolvable} when $\alpha_2\neq 0$ or Theorem \ref{3solvablecase1} when $\gamma_1+1\neq 0$ or Theorem \ref{3solvablecase2} when $\gamma_1+1\neq 0$ or Theorem \ref{3solvablecase3} when $\gamma_0+\gamma_1+1\neq 0$ hold. If $\mathfrak{a}'$ is the unique maximal ideal of the local algebra $A$ then the ideal $(\mathfrak{a}')$ is the maximal ideal of $V_B$. In addition, $V_B/(\mathfrak{a}')$ is isomorphic to the rank one Heisenberg vertex algebra $M(1)$. 
\end{thm}
\begin{proof} Assume that $B$ is a vertex $A$-algebroid that satisfies assumptions in either Theorem \ref{Bcyclicsolvable} when $\alpha_2\neq 0$ or Theorem \ref{3solvablecase1} when $\gamma_1+1\neq 0$ or Theorem \ref{3solvablecase2} when $\gamma_1+1\neq 0$ or Theorem \ref{3solvablecase3} when $\gamma_0+\gamma_1+1\neq 0$ hold. 

Let $\mathfrak{a}'$ be the unique maximal ideal of $A$. Recall that $A=\mathbb{C}1_A\oplus \mathfrak{a}'$. In addition, $B=\mathbb{C} b\oplus\partial(A)$. Let $\bar{b}\in\mathbb{C}^{\times} b$ such that $\bar{b}_1\bar{b}=1_A+a'$ where $a'\in \mathfrak{a}'$. Notice that $(V_B/(\mathfrak{a}'))_0=\mathbb{C}(1_A+(\mathfrak{a}')$ and for $n\geq 1$ $(V_B/(\mathfrak{a}'))_n=Span_{\mathbb{C}}\{\bar{b}(-n_1)...\bar{b}(-n_k){\bf 1}+(\mathfrak{a}')~|~n_1\geq ....\geq n_k\geq 1,~n_1+...+n_k=n\}$. Note ${\bf 1}=1_A$.

Let $\mathfrak{h}$ be a one-dimensional abelian Lie algebra spanned by $h$ with a linear form $\langle\cdot,\cdot\rangle$ such that $\langle h,h\rangle=1$. Let $f:M(1)\rightarrow V_B/(\mathfrak{a}')$ be a linear map defined by 
\begin{eqnarray*}
f({\bf 1})&=&1_A+(\mathfrak{a}'),\\
f(h(-n_1)...h(-n_k){\bf 1})&=&\bar{b}(-n_1)...\bar{b}(-n_k)1_A+(\mathfrak{a}'),\text{ and }\\
fY(h(-1){\bf 1},x)&=&Y(f(h(-1){\bf 1}),x)f.
\end{eqnarray*}
Notice that we have $ fY({\bf 1},x)=Y(1_A,x)f$ and $f$ is onto. Now, we will show that $f$ is a vertex algebra homomorphism. We set 
$C(f)=\{v\in M(1)~|~Y(f(v),x)f=fY(v,x)\}$. To show that $f$ is a vertex algebra homomorphism. It is enough to show that $C(f)=M(1)$. Clearly, ${\bf 1}, h(-1){\bf 1}\in C(f)$. Let $u,v\in C(f)$ and $n\in\mathbb{Z}$. Using the fact that $Y(u_nv,x)=\Res_{x_1}((x_1-x)^nY(u,x_1)Y(v,x)-(-x+x_1)^nY(v,x)Y(u,x_1))$, we then have that 
\begin{eqnarray*}
    &&fY(u_nv,x)\\
    &=&\Res_{x_1}((x_1-x)^nfY(u,x_1)Y(v,x)-(-x+x_1)^nfY(v,x)Y(u,x_1))\\
    &=&\Res_{x_1}((x_1-x)^nY(f(u),x_1)fY(v,x)-(-x+x_1)^nY(f(v),x)fY(u,x_1))\\
    &=&\Res_{x_1}((x_1-x)^nY(f(u),x_1)Y(f(v),x)f-(-x+x_1)^nY(f(v),x)Y(f(u),x_1)f)\\
    &=&Y(f(u)_nf(v),x)f.
\end{eqnarray*}
Hence, $u_nv\in C(f)$. In addition, $C(f)=M(1)$ and $f$ is a homomorphism. Since $M(1)$ is simple, we can conclude further that $f$ is an isomorphism and $(\mathfrak{a}')$ is a maximal ideal of $V_B$. 
\end{proof}


\end{document}